\documentclass{amsart}
\usepackage{enumerate}
\usepackage{amsthm,amssymb}
\usepackage{float}
\usepackage{xcolor}
\usepackage{mathrsfs}
\usepackage{multirow}
\usepackage{amsmath}
\usepackage{mathtools}
\usepackage{graphicx}
\usepackage{color}
\usepackage{url}

\newtheorem{theorem}{Theorem}[section]

\newtheorem{lemma}[theorem]{Lemma}

\newtheorem{definition}[theorem]{Definition}
\theoremstyle{definition}

\newtheorem{exmp}[theorem]{Example}
\newtheorem{alg}[theorem]{Algorithm}

\newcommand{\rvline}{\hspace*{-\arraycolsep}\vline\hspace*{-\arraycolsep}}
\newcommand{\reff}[1]{(\ref{#1})}

\numberwithin{equation}{section}

\def\mA{\mathcal{A}}

\def\mF{\mathcal{F}}
\def\mH{\mathcal{H}}
\def\mX{\mathcal{X}}
\def\mE{\mathcal{E}}
\def\eps{\epsilon}

\def\bbC{\mathbb{C}}
\def\bbF{\mathbb{F}}

\def\bold{\mathbf}

\newcommand{\rank}{\operatorname{rank}}

\newcommand{\be}{\begin{equation}}
\newcommand{\ee}{\end{equation}}
\newcommand{\beq}{\begin{equation}}
\newcommand{\eeq}{\end{equation}}
\newcommand{\baray}{\begin{array}}
\newcommand{\earay}{\end{array}}

\newcommand{\bbm}{\begin{bmatrix}}
\newcommand{\ebm}{\end{bmatrix}}

\begin{document}

\title[Low Rank Tensor Decompositions and Approximations]
{Low Rank Tensor Decompositions and Approximations}

\author{Jiawang Nie, Li Wang, and Zequn Zheng}
\address{Jiawang Nie and Zequn Zheng
Department of Mathematics, University of California San Diego,
9500 Gilman Drive, La Jolla, CA, USA, 92093.}
\address{Li Wang
	Department of Mathematics, University of Texas at Arlington,
	411 South Nedderman Drive, Arlington, TX, 76019.}
\email{njw@math.ucsd.edu, li.wang@uta.edu, zez084@ucsd.edu}

\subjclass[2010]{15A69,65F99}

\keywords{tensor, decomposition, rank, approximation,
generating polynomial}

\begin{abstract}
There exist linear relations among tensor entries of low rank tensors.
These linear relations can be expressed by multi-linear polynomials,
which are called generating polynomials.
We use generating polynomials to
compute tensor rank decompositions and low rank tensor approximations.
We prove that this gives a quasi-optimal
low rank tensor approximation
if the given tensor is sufficiently close to a low rank one.
\end{abstract}

\maketitle

\section{Introduction}
Let $m$ and $n_1, \ldots, n_m$ be positive integers.
A tensor $\mF$ of order $m$ and dimension $n_1 \times \cdots \times n_m$
can be labelled such that
\[
\mF  = (\mF_{i_1,\ldots,i_m})_{1\leq i_1 \leq n_1, \ldots, 1 \leq i_m \leq n_m}.
\]
Let $\bbF$ be a field (either the real field $\mathbb{R}$ or the complex field $\mathbb{C}$).
The space of all tensors of order
$m$ and dimension $n_1 \times \cdots \times n_m$
with entries in $\bbF$, is denoted as
$\bbF^{n_1\times\cdots\times n_m}$.
For vectors $v_i \in \mathbb{F}^{n_i}$, $i=1,\ldots,m$,
their outer product $v_1 \otimes \ldots \otimes v_m$
is the tensor in $\mathbb{F}^{n_1\times\cdots\times n_m}$  such that
\[
(v_1 \otimes \cdots \otimes v_m)_{i_1,\ldots,i_m}
\, = \, (v_1)_{i_1} \cdots (v_m)_{i_m}
\]
for all labels in the corresponding range. A tensor like
$v_1 \otimes \ldots \otimes v_m$ is called a rank-$1$ tensor.
For every tensor $\mF \in \mathbb{F}^{n_1\times\cdots\times n_m}$,
there exist vector tuples $(v^{s,1}, \ldots, v^{s,m})$,
$s=1,\ldots,r$,~$v^{s,j} \in \bbF^{n_j}$, such that
\begin{align}   \label{eq:ransum_pure}
\mF = \sum_{s=1}^r v^{s,1} \otimes \cdots \otimes v^{s, m}.
\end{align}
The smallest such $r$ is called the $\bbF$-rank of $\mF$,
for which we denote $\text{rank}_{\bbF}(\mF)$.
When $r$ is minimum, the equation \reff{eq:ransum_pure}
is called a rank-$r$ decomposition over the field $\mathbb{F}$.
In the literature, this rank is sometimes referenced as
the candecomp-parafac (CP) rank.
We refer to \cite{Land12,Lim13} for various notions of tensor ranks.
Recent work for tensor decompositions can be found in
\cite{BreVan18,dLMV04,dLa06,Domanov2014,GPSTD,SvBdL13,telen2021normal}.
Tensors are closely related to polynomial optimization
\cite{CDN14,FNZ18,nie2014semidefinite,NieZhang18,NYZ18}.
Various applications of tensors can be found in
\cite{guo2021learning,kolda2009tensor,nie2020hermitian}.
Throughout the paper, we use the Hilbert-Schmidt norm for tensors:
\[
\| \mF \| = \sqrt{
\sum_{\substack{ 1\leq i_j \leq n_j, 1 \le j \le m } }
|\mF_{i_1,\ldots,i_m}|^2
} .
\]

The low rank tensor approximation (LRTA) problem is to find a low rank tensor
that is close to a given one.
This is equivalent to a nonlinear least square optimization problem.
For a given tensor $\mF \in  \mathbb{F}^{n_1\times\cdots\times n_m}$
and a given rank $r$, we look for $r$ vector tuples
$v^{(s)}: = (v^{s,1}, \ldots, v^{s,m})$, $s=1,\ldots,r$, such that
\[
\mF  \, \approx \, \sum\limits_{s = 1}^r v^{s,1} \otimes \cdots \otimes v^{s, m},
\quad  v^{s,j} \in \mathbb{F}^{n_j} .
\]
This requires to solve the following nonlinear least squares optimization
\begin{align} \label{nonlinear:least:square:F}
\min\limits_{ v^{(s,j)}  \in  \mathbb{F}^{n_j}, \, j=1,\ldots,m  }
\big\|\mF - \sum\limits_{s = 1}^r v^{s,1} \otimes \cdots \otimes v^{s, m}\big\|^2.
\end{align}
When $r = 1$, the best rank-$1$ approximating tensor always exists
and it is equivalent to computing the spectral norm
\cite{FriWang20,QiHu19}.
When $r>1$, the best rank-$r$ tensor approximation may not exist \cite{de2008tensor}.
Classically used methods for solving low rank tensor approximations are
the alternating least squares (ALS) method \cite{CLdA09,MYY20,YYang20},
higher order power iterations \cite{de2000best},
semidefinite relaxations \cite{CDN14,nie2014semidefinite},
SVD based methods \cite{guan2018convergence},
optimization based methods \cite{FriTam14,SvBdL13}.
We refer to \cite{BreVan18,LowRankSymmetric}
for recent work on low rank tensor approximations.
%
%

\subsection*{Contributions}

In this paper, we extend the generating polynomial method in
\cite{GPSTD,LowRankSymmetric}
to compute tensor rank decompositions and low rank tensor approximations
for nonsymmetric tensors.
First, we estimate generating polynomials by solving linear least squares.
Second, we find their approximately common zeros,
which can be done by computing eigenvalue decompositions.
Third, we get a tensor decomposition from their common zeros,
by solving linear least squares. To find a low rank tensor approximation,
we first apply the decomposition method to obtain a low rank approximating tensor
and then use nonlinear optimization methods to improve the approximation.
Our major conclusion is that if the tensor to be approximated is sufficiently close
to a low rank one, then the obtained low rank tensor is a quasi-optimal low rank approximation.
The proof is based on perturbation analysis of linear least squares and eigenvalue decompositions.

%
The paper is organized as follows. In Section~\ref{sec:pre},
we review some basic results about tensors.
In Section~\ref{sc:gp}, we introduce the concept of generating polynomials
and study their relations to tensor decompositions.
In Section~\ref{sec:lowTD}, we give an algorithm for computing
tensor rank decompositions for low rank tensors.
In Section~\ref{sec:lowTA}, we give an algorithm for computing low rank approximations.
The approximation error analysis is also given.
In Section~\ref{sec:num}, we present numerical experiments.
Some conclusions are made in Section~\ref{sec:con}.

\section{preliminary}
\label{sec:pre}

\subsection*{Notation}
The symbol $\mathbb{N}$ (resp., $\mathbb{R}$, $\mathbb{C}$)
denotes the set of nonnegative integers (resp., real, complex numbers).
For an integer $r>0$, denote the set $[r] \coloneqq \{1, \ldots, r\}$.
Uppercase letters (e.g., $A$) denote matrices,
$A_{ij}$ denotes the $(i,j)$th entry of the matrix $A$,
and Curl letters (e.g., $\mathcal{F}$) denote tensors,
$\mathcal{F}_{i_1,...,i_m}$ denotes the $(i_1,...,i_m)$th entry of the tensor $\mathcal{F}$.
For a complex matrix $A$, $A^T$ denotes its transpose
and $A^*$ denotes its conjugate transpose.
The Kruskal rank of $A$, for which we denote $\kappa_A$,
is the largest number $k$ such that every set of $k$ columns of $A$
is linearly independent.
For a vector $v$, the $(v)_i$ denotes its $i$th entry and
$\mbox{diag}(v)$ denotes the square diagonal matrix whose
diagonal entries are given by the entries of $v$.
The subscript $v_{s:t}$ denotes the subvector of $v$
whose label is from $s$ to $t$.
For a matrix $A$, the subscript notation $A_{:,j}$ and $A_{i,:}$
respectively denote its $j$th column and $i$th row.
Similar subscript notation is used for tensors.
For two matrices $A,B$, their classical Kronecker product is denoted as $A \boxtimes B$.
For a set $S$, its cardinality is denoted as $|S|$.

For a tensor decomposition for $\mF$ such that
\be \label{mATD:usj}
\mF \, = \, \sum_{s=1}^r u^{s,1}\otimes  u^{s,2} \otimes...\otimes u^{s,m}.
\ee
we denote the matrices
\[
U^{(j)} \, = \, [u^{1,j},..., u^{r,j}],
\quad\, j =1, \ldots, m.
\]
The $U^{(j)}$ is called the $j$th decomposing matrix for $\mF$.
For convenience of notation, we denote that
\begin{align*}
U^{(1)} \circ \cdots \circ U^{(m)} =
\sum_{i=1}^r (U^{(1)})_{:,i} \otimes\ldots \otimes (U^{(m)})_{:,i}.
\end{align*}
Then the above tensor decomposition is equivalent to
$\mF = U^{(1)} \circ \cdots \circ U^{(m)}$.

For a matrix $V \in \mathbb{C}^{p \times n_t}$,
define the matrix-tensor product
\[
\mA \coloneqq   V \times_t \mF
\]
is a tensor in
$\mathbb{C}^{n_1\times ...\times n_{t-1} \times p \times n_{t+1} \times ...\times n_m}$
such that the $i$th slice of $\mA$ is
\[
\mA_{i_1,...,i_{t-1},:,i_{t+1},...,i_m} =
V \mF_{i_1,...,i_{t-1},:,i_{t+1},...,i_m}.
\]

\subsection{Flattening matrices}
\label{ssc:flat}

We partition the dimensions $n_1,n_2,...,n_m$ into two disjoint groups
$I_1$ and $I_2$ such that the difference
\[
\Big| \prod_{i \in I_1 } n_{i} - \prod_{j \in I_2} n_{j} \Big|
\]
is minimum. Up to a permutation, we write that
$I_1 = \{n_{1}, \ldots, n_{k}  \}$,
$I_2 = \{n_{k+1}, \ldots, n_{m} \}.$
For convenience, denote that
$$
\begin{array}{l}
I = \left\{\left(\imath_{1}, \ldots, \imath_{k} \right):
       1 \leq \imath_{j} \leq n_{j}, j=1, \ldots, k \right\} ,\\
J =\left\{\left(\imath_{k+1}, \ldots, \imath_{m}\right):
       1 \leq \imath_{j} \leq n_{j}, j=k+1, \ldots, m\right\}.
\end{array}
$$
For a tensor $\mF \in \bbC^{n_{1}\times \ldots\times n_{m}}$,
the above partition gives the flattening matrix
\begin{align}
    \label{ssc:flat:square}
\mbox{Flat} (\mF ) \coloneqq
\left(\mF_{\imath, \jmath}\right)_{\imath \in I, \jmath \in J}.
\end{align}
This gives the most square flattening matrix for $\mF$.
Let $\sigma_{r}$ denote the closure of all rank-$r$ tensors in
$\bbC^{n_{1}\times \ldots\times n_{m}}$,
under the Zariski topology (see \cite{cox2013ideals}).
The set $\sigma_r$ is an irreducible variety of $\bbC^{n_1 \times \cdots \times n_{m}}$.
For a given tensor $\mF \in \sigma_r$, it is possible that $\rank(\mF) > r.$
This fact motivates the notion of border rank:
\be
\rank_{B}(\mF)=\min \left\{r: \mF \in \sigma_r \right\},
\ee
For every tensor $\mF\in \mathbb{C}^{n_{1}\times \ldots \times n_{m}}$,
one can show that
\be     \label{rank_ineq}
\rank \mbox{Flat}(\mF) \leq \rank_{B}(\mF)
\leq \rank(\mF) .
\ee
A property $\mathtt{P}$ is said to hold $generically$ on
$\sigma_{r}$ if $\mathtt{P}$ holds on a Zariski open subset $T$ of $\sigma_r$.
For such a property $\mathtt{P}$, each $u \in T$ is called a generic point.
Interestingly, the above three ranks are equal for generic points of
$\sigma_r$ for a range of values of $r$.

\begin{lemma}
\label{cat_mat_lemma}
Let $s$ be the smaller dimension of the matrix $\mbox{Flat}(\mF)$.
For every $r \leq s$, the equalities
\be  \label{rank_eq}
 \rank \mbox{Flat}(\mF) \,= \, \rank_B(\mF) \,= \,  \rank (\mF)
\ee
hold for tensors $\mF$ in a Zariski open subset of $\sigma_r$.
\end{lemma}
\begin{proof}
Let $\phi_{1}, \ldots, \phi_{\ell}$ be the $r \times r$ minors of the matrix
\begin{align}
 \mbox{Flat} (\sum_{i=1}^r x^{i, 1} \otimes \cdots \otimes x^{i, m} ).
\end{align}
They are homogeneous polynomials in $x^{i, j}(i=1, \ldots, r, j=1, \ldots, m) .$
Let $x$ denote the tuple $\left(x^{1,1}, x^{1,2}, \ldots, x^{r, m}\right) .$
Define the projective variety in $\mathbb{P}^{ r\left(n_{1}+\cdots+n_{m}\right)-1}$
\begin{align}
Z = \left\{x: \phi_{1}(x)=\cdots=\phi_{\ell}(x)=0\right\}.
\end{align}
Then $Y:=\mathbb{P}^{r\left(n_{1}+\cdots+n_{m}\right)-1} \backslash Z$
is a Zariski open subset of full dimension.
Consider the polynomial mapping $\pi: Y \rightarrow \sigma_r$,
\begin{align}
\left(x^{1,1}, x^{1,2}, \ldots, x^{r, m}\right) \mapsto
\sum_{i=1}^{r}\left(x^{i, 1}\right) \otimes \cdots \otimes\left(x^{i, m}\right) .
\end{align}
The image $\pi(Y)$ is dense in the irreducible variety $\sigma_{r}.$
So, $\pi(Y)$ contains a Zariski open subset $\mathscr{Y}$
of $\sigma_r$ (see \cite{schafarevich1988basic}).
For each $\mF \in \mathscr{Y}$, there exists $u \in Y$ such that $\mF = \pi(u)$.
Because $u \notin Z$, at least one of $\phi_{1}(u), \ldots, \phi_{\ell}(u)$ is nonzero,
and hence $\rank \mbox{Flat}(\mF) \geq r .$ By \eqref{rank_ineq}, we know
\eqref{rank_eq} holds for all $\mF \in \mathscr{Y}$
since $\rank (\mF) \leq r .$ Since $\mathscr{Y}$ is a Zariski open subset of $\sigma_r$,
the lemma holds.
\end{proof}

By Lemma \ref{cat_mat_lemma}, if $r \leq s$ and $\mF$ is a generic tensor
in $\sigma_{r}$, we can use $\rank \mbox{Flat}(\mF)$ to estimate  $\rank (\mF)$.
However, for a generic $\mathcal{F} \in \mathbb{C}^{n_{1} \times \cdots \times n_{m}}$
such that $\rank \mbox{Flat}(\mF)=r$, we cannot conclude $\mF \in \sigma_r$.

\subsection{Reshaping of tensor decompositions}
\label{ssc:reshape}

A tensor $\mF$ of order greater than $3$ can be reshaped to
another tensor $\widehat{\mF}$ of order $3$. A tensor decomposition of
$\widehat{\mF}$ can be converted to a decomposition for $\mF$
under certain conditions.
In the following, we assume a given tensor $\mF$ has the decomposition \eqref{mATD:usj}.
Suppose the set $\{1,\ldots,m\}$ is partitioned into $3$ disjoint subsets
 \[
\{1,\ldots,m\}  \, = \,
I_1 \cup I_2 \cup I_3.
\]
Let $p_i=|I_i| \mathrm{~for~} i=1,2,3$.
For the reshaped vectors
\be \label{KRprod:vwz}
\left\{ \baray{rcll}
 w^{s,1} & = & u^{s,i_1} \boxtimes \cdots \boxtimes u^{s,i_{p_1}}
&\mathrm{~for~} I_1 \, = \, \{ i_1, \ldots, i_{p_1} \},\\
w^{s,2} &=& u^{s,j_1} \boxtimes \cdots \boxtimes u^{s,j_{p_2}}
&\mathrm{~for~} I_2 \, = \, \{ j_1, \ldots, j_{p_2} \}, \\
w^{s,3} & = & u^{s,k_1} \boxtimes \cdots \boxtimes u^{s,k_{p_3}}
&\mathrm{~for~} I_3 \, = \, \{ k_1, \ldots, k_{p_3} \},
\earay \right.
\ee
we get the following tensor decomposition
\be \label{outer_product_form}
\widehat{\mF} \, = \, \sum_{s=1}^r
w^{s,1}\otimes w^{s,2} \otimes w^{s,3} .
\ee
Conversely, for a decomposition like \reff{outer_product_form} for $\widehat{\mF}$,
if all $w^{s,1}, w^{s,2}, w^{s,3}$ can be expressed
as rank-$1$ products as in \reff{KRprod:vwz},
then the equation \reff{outer_product_form} can be reshaped to
a tensor decomposition for $\mF$ as in \reff{mATD:usj}.
When the flattened tensor $\widehat{\mF}$ satisfies some conditions,
the tensor decomposition of $\widehat{\mF}$ is unique.
For such a case, we can obtain a tensor decomposition for $\mF$
through the decomposition \reff{outer_product_form}.
A classical result about the uniqueness is
the Kruskal's criterion \cite{Kruskal1977ThreeWayArrays}.

\begin{theorem} \label{thm:Kruskal}
(Kruskal's Criterion, \cite{Kruskal1977ThreeWayArrays})
Let $\mathcal{F} = U^{(1)} \circ U^{(2)}  \circ U^{(3)}$ be a tensor
with each $U^{(i)} \in \bbC^{n_i \times r}$.
Let $\kappa_i$ be the Kruskal rank of $U^{(i)}$, for $i=1,2,3$.
If \[ 2r +2 \leq \kappa_1+\kappa_2+\kappa_3,  \]
then $\mathcal{F}$ has a unique rank-$r$ tensor decomposition.
\end{theorem}

The Kruskal's Criterion
can be generalized for more range of $r$ as in \cite{Chiantini2017}.
Assume the dimension $n_1 \geq n_2 \geq n_3 \geq 2$
and the rank $r$ is such that
\[
2r + 2 \leq  \mathrm{min}(n_1,r)+\mathrm{min}(n_2,r)+\mathrm{min}(n_3,r),
\]
or equivalently, for $\delta =n_2+n_3-n_1-2$, $r$ is such that
\[
r \leq n_1 +\mathrm{min}\{ \frac{1}{2}\delta,\delta \}.
\]
If $\mF$ is a generic tensor of rank $r$ as above
in the space $\bbC^{n_1 \times n_2 \times n_3}$,
then $\mF$ has a unique rank-$r$ decomposition.
There following is a uniqueness result for reshaped tensor decompositions.

\begin{theorem}  \label{fla}
(Reshaped Kruskal Criterion, \cite[Theorem~4.6]{Chiantini2017})
For the tensor space
$\bbC^{n_1 \times n_2 \times \cdots \times n_m}$ with $m \ge 3$,
let $I_1 \cup I_2 \cup I_3 =\{1,2,...,m\}$ be a union of disjoint sets
and let
\[
p_1 = \prod_{i \in I_1} n_i, \quad
p_2 = \prod_{j \in I_2} n_j, \quad
p_3 = \prod_{k \in I_3} n_k .
\]
Suppose $p_1 \geq p_2 \geq p_3$ and let $\delta=p_2+p_3-p_1-2$.
Assume
\be     \label{generlized_kr}
r \leq p_1+\mathrm{min}\{\frac{1}{2}\delta,\delta \} .
\ee
If $\mF$ is a generic tensor of rank $r$
in $\bbC^{n_1 \times n_2 \times \cdots \times n_m}$,
then the reshaped tensor $\widehat{\mF}\in \mathbb{C}^{p_1 \times p_2 \times p_3}$
as in \reff{outer_product_form}
has a unique rank-$r$ decomposition.
\end{theorem}

\section{Generating Polynomials}
\label{sc:gp}

Generating polynomials can be used to compute tensor rank decompositions.
We consider tensors whose ranks $r$ are not bigger than the highest dimension,
say, $r\leq n_1$ where $n_1$ is the biggest of $n_1, \ldots, n_m$.
Denote the indeterminate vector variables
\[
\bold{x_1} = (x_{1,1},...x_{1,n_1}),
\, \ldots, \,
\bold{x_m} = (x_{m,1},...x_{m,n_m}).
\]
A tensor in $\bbC^{n_1 \times \cdots \times n_m}$
can be labelled by monomials like $x_{1,i_1}x_{2,i_2}...x_{m,i_m}$. Let
\begin{equation}
\begin{array}{rcl}
\mathbb{M} & \coloneqq &
\big\{ x_{1,i_1}...x_{m,i_m} \; | \;1\leq i_j \leq n_j \big\} .
\end{array}
\end{equation}
For a subset $J \subseteq \{1,2,...,m \}$,   denote that
\be 	\label{m_j_def}
\boxed{
\begin{array}{rcl}
	J^c  & \coloneqq & \{1,2,...,m \} \backslash J,  \\
	\mathbb{M}_J  & \coloneqq & \big\{ x_{1,i_1}...x_{m,i_m} \; | \;
      x_{j,i_j} = 1\, ~\forall \,   j\in J^c \big\}, \\
   \mathcal{M}_J & \coloneqq & {\rm span} \{\mathbb{M}_J\}.
\end{array}
}
\ee
A label tuple $(i_1,\ldots,i_m)$ is uniquely determined by the monomial
$x_{1,i_1} \cdots x_{m,i_m}$.
So a tensor $\mF \in \mathbb{C}^{n_1\times \ldots \times n_m}$
can be equivalently labelled by monomials such that
\begin{align} \label{label:monomial}
\mF_{x_{1,i_1}\ldots x_{m,i_m}}   \coloneqq   \mF_{i_1,\ldots,i_m}.
\end{align}
With the new labelling by monomials, define the bi-linear product
\begin{align}  	\label{polyten}
\langle \sum_{\mu \in \mathbb{M}}c_{\mu}\mu,\mF \rangle
 \, \coloneqq  \, \sum_{\mu \in \mathbb{M}}c_{\mu}\mF_{\mu}.
\end{align}
In the above, each $c_{\mu}$ is a scalar
and $\mathcal{F}$ is labelled by monomials as in (\ref{label:monomial}).

\begin{definition}  \label{def:gp}
(\cite{NieLR14,NWZ22})
For a subset $J \subseteq \{1,2,...,m\}$ and a tensor
$\mathcal{F} \in \mathbb{C}^{n_1 \times \cdots \times n_m}$,
a polynomial $p \in \mathcal{M}_J$
is called a generating polynomial for $\mathcal{F}$ if
\begin{equation}  		\label{inner_p}
\langle pq,\mathcal{F} \rangle =0 \quad
\mbox{for all} \, \, q \in \mathbb{M}_{J^c} .
\end{equation}
\end{definition}

The following is an example of generating polynomials.

\begin{exmp}
Consider the cubic order tensor
$\mathcal{F} \in \mathbb{C}^{3 \times 3 \times 3}$ given as
\begin{gather*}
\begin{bmatrix}
\mathcal{F}_{:,:,1} & \rvline & \mathcal{F}_{:,:,2} & \rvline & \mathcal{F}_{:,:,3} \\
\end{bmatrix}=
\begin{bmatrix}
\begin{matrix*}[r]
11 & 20 & 10\\
7 & 10 & 5 \\
12 & 18 & 9 \\
\end{matrix*} & \rvline &
\begin{matrix*}[r]
15 & 24 & 12 \\
15 & 18 & 9 \\
24 & 30 & 15 \\
\end{matrix*} &\rvline &
\begin{matrix*}[r]
7 & 10 & 5 \\
9 & 10 & 5 \\
14 & 16 & 8 \\
 \end{matrix*} \\
\end{bmatrix}.
\end{gather*}
The following is a generating polynomial for $\mathcal{F}$:
\[
p \coloneqq (2x_{1,1}-x_{1,2})(2x_{2,1}-x_{2,2}).
\]Note that $p \in \mathcal{M}_{\{1,2 \}}$
and for each $i_3 = 1, 2, 3$
\begin{align*}
p \cdot x_{3,i_3}  =  (2x_{1,1}-x_{1,2})(2x_{2,1}-x_{2,2})  x_{3,i_3}.
\end{align*}
One can check that for each  $i_3 = 1, 2, 3$,
\[
4\mathcal{F}_{1,1,i_3}-2\mathcal{F}_{1,2,i_3}-
2\mathcal{F}_{2,1,i_3}+\mathcal{F}_{2,2,i_3}=0.
\]
This is because
$\begin{bmatrix}4 & -2 & -2 & 1 \end{bmatrix}$ is orthogonal to
\[
\begin{bmatrix}
\mathcal{F}_{1,1,i_3} & \mathcal{F}_{1,2,i_3} &
\mathcal{F}_{2,1,i_3} & \mathcal{F}_{2,2,i_3}
\end{bmatrix}
\]
for each $i_3=1,2,3$.
\end{exmp}

Suppose the rank $r \le n_1$ is given.
For convenience of notation, denote the label set
\begin{align}  \label{set:J}
J  \coloneqq  \{(i,j,k):1\leq i \leq r, ~2 \leq j \leq m, ~2\leq k \leq n_j \}.
\end{align}
For a matrix $G \in \mathbb{C}^{[r]\times J}$
and a triple $\tau =(i,j,k) \in J$, define the bi-linear polynomial
\begin{align}  	\label{phi_def}
\phi[G,\tau](x) \coloneqq  \sum_{\ell=1}^{r}
G(\ell,\tau)x_{1,\ell}x_{j,1}-x_{1,i}x_{j,k}
\,\, \in \, \mathcal{M}_{\{1,j\}}.
\end{align}
The rows of $G$ are labelled by $\ell=1,2,...,r$
and the columns of $G$ are labelled by $\tau \in J$.
We are interested in $G$ such that
$\phi[G,\tau]$ is a generating polynomial for a tensor
$\mathcal{F} \in \mathbb{C}^{n_1 \times n_2 \times \ldots \times n_m}$.
This requires that
\begin{align*}
\langle \phi[G,\tau] \cdot \mu ,\mathcal{F} \rangle=0 \,
\quad \mbox{for all} \, \, \mu \in \mathbb{M}_{\{1,j\}^c}.
\end{align*}
The above is equivalent to the equation
($\mathcal{F}$ is labelled as in (\ref{label:monomial}))
\begin{equation}  	\label{linear_eq_g}
\sum_{\ell=1}^{r} G(\ell,\tau)\mathcal{F}_{x_{1,\ell} \cdot \mu}
=\mathcal{F}_{x_{1,i}  x_{j,k}  \cdot \mu} .
\end{equation}

\begin{definition} (\cite{NieLR14,NWZ22})
When \eqref{linear_eq_g} holds for all $\tau \in J$,
the $G$ is called a generating matrix for $\mathcal{F}$.
\end{definition}

For given $G$, $j \in \{ 2, \ldots, m \}$ and $k \in \{2,\ldots,n_j \}$,
we denote the matrix
\begin{equation}
M^{j,k}[G] \coloneqq \begin{bmatrix}
	G(1,(1,j,k)) & G(2,(1,j,k)) & \dots & G(r,(1,j,k)) \\
	G(1,(2,j,k)) & G(2,(2,j,k)) & \dots & G(r,(2,j,k)) \\
	\vdots & \vdots & \ddots & \vdots \\
	G(1,(r,j,k)) & G(2,(r,j,k)) & \dots & G(r,(r,j,k)) \\
\end{bmatrix} .
\label{mjk}
\end{equation}
For each $(j,k)$, define the matrix/vector
\be  \label{def:A_b}
\left\{ \baray{rcl}
 A[\mF,j]  &  \coloneqq &
\Big( \mF_{x_{1,\ell}\cdot \mu}
     \Big)_{\mu \in \mathbb{M}_{\{1,j\}^c}, 1\leq \ell \leq r} , \\
 b[\mF,j,k] & \coloneqq &
     \Big(\mF_{x_{1,\ell} \cdot x_{j,k}\cdot \mu}
      \Big)_{\mu \in \mathbb{M}_{\{1,j\}^c}, 1\leq \ell \leq r} .	
\earay  \right.
\ee
The equation \eqref{linear_eq_g} is then equivalent to
\be \label{linear_eq}
A[\mF,j] (M^{j,k}[G])^T \, = \, b[\mF,j,k].
\ee
The following is a useful property
for the matrices $M^{j,k}[G]$.

\begin{theorem} (\cite{NieLR14,NWZ22})  \label{decom_to_Mjk}
Suppose $\mF =\sum_{s=1}^r u^{s,1} \otimes ... \otimes u^{s,m}$
for vectors $u^{s,j} \in \bbC^{n_j}$.
If $r \leq n_1$, $(u^{s,2})_1...(u^{s,m})_1 \neq 0$,
and the first $r$ rows of the first decomposing matrix
\[
U^{(1)} \coloneqq [u^{1,1} \, \, \cdots \,\,  u^{r,1} ]
\]
are linearly independent,
then there exists a $G$ satisfying \eqref{linear_eq}
and satisfying (for all $j \in \{ 2, \ldots, m \}$,
$k \in \{2,\ldots,n_j \}$ and $s=1,\ldots,r$)
\begin{align}  	\label{Mjk[G]:eigeqn}
M^{j,k}[G]\cdot (u^{s,1})_{1:r}=
(u^{s,j})_k\cdot (u^{s,1})_{1:r}.
\end{align}
\end{theorem}

\section{Low rank tensor decompositions}
\label{sec:lowTD}

%
Without loss of generality,   assume
the dimensions are decreasing as
$n_1 \ge n_2 \ge \cdots \ge n_m.$
We discuss how to compute tensor decomposition
for a tensor $\mF \in \bbC^{n_1 \times \cdots \times n_m}$
when the rank $r$ is not bigger than the highest dimension,
i.e., $r \le n_1$.
As in Theorem~\ref{decom_to_Mjk}, the decomposing vectors $(u^{s,1})_{1:r}$
are common eigenvectors of the matrices $M^{j,k}[G]$,
with $(u^{s,j})_k$ being the eigenvalues respectively.
This implies that the matrices $M^{j,k}[G]$
are simultaneously diagonalizable.
This property can be used to compute tensor decompositions.

Suppose $G$ is a matrix such that \eqref{linear_eq} holds and
$M^{j,k}[G]$ are simultaneously diagonalizable. That is,
there is an invertible matrix $P \in \bbC^{r \times r}$
such that all the products $P^{-1}  M^{j,k}[G]  P$
are diagonal for all $j=2,\ldots, m$ and for all $k = 2, \ldots, n_j$.
Suppose $M^{j,k}[G]$ are diagonalized such that
\be \label{diag:eig:P}
P^{-1}  M^{j,k}[G]  P \, = \,
\mbox{diag}[\lambda_{j,k,1},\lambda_{j,k,2},\ldots,\lambda_{j,k,r}]
\ee
with the eigenvalues $\lambda_{j,k,s}$.
For each $s =1,\ldots,r$ and $j=2,\ldots, m$,
denote the vectors
\be \label{usj:j2:m}
u^{s,j}  \, \coloneqq \,  (1, \lambda_{j,2, s}, \ldots, \lambda_{j,n_j,s} ).
\ee
When $\mF$ is rank-$r$, there exist vectors
$u^{1,1}, \ldots, u^{r,1} \in \bbC^{n_1}$ such that
\be  \label{exist:u(s1)}
\mF  =\sum_{s=1}^r u^{s,1} \otimes u^{s,2} \otimes \cdots  \otimes u^{s,m} .
\ee
The vectors $u^{s,1}$ can be found by solving linear equations
after $u^{s,j}$ are obtained for $j=2,\ldots,m$ and $s=1,\ldots,r$.
The existence of vectors $u^{s,1}$ satisfying the tensor decomposition
(\ref{exist:u(s1)}) is shown in the following theorem.

\begin{theorem}	\label{onetoone}
Let $\mathcal{F} =V^{(1)} \circ ... \circ V^{(m)}$ be a rank-r tensor,
for matrices $V^{(i)} \in \bbC^{n_i \times r}$,
such that the first $r$ rows of $V^{(1)}$ are linearly independent.
Suppose $G$ is a matrix satisfying \eqref{linear_eq} and
$P \in \bbC^{r \times r}$ is an invertible matrix such that all matrix products
$P^{-1} \cdot M^{j,k}[G] \cdot P$
are simultaneously diagonalized as in \eqref{diag:eig:P}.
For $j=2,\ldots,m$ and $s=1,\ldots, r$,
let $u^{s,j}$ be vectors given as in  \eqref{usj:j2:m}.
Then, there must exist vectors $u^{1,1}, \ldots, u^{r,1} \in \bbC^{n_1}$
such that the tensor decomposition \eqref{exist:u(s1)} holds.
\end{theorem}
\begin{proof}
Since the matrix $P = \bbm p_1 & \cdots & p_r \ebm$
is invertible, there exist scalars $c_1,\ldots, c_r \in \bbC$ such that
\be  \label{proof}
\mF_{1:r,1,...,1}  = c_1 p_1 +c_2 p_2
+ \cdots + c_r p_r .
\ee
Consider the new tensor
\[
\mH \, \coloneqq \, \sum_{s=1}^r c_s p_s \otimes
u^{s,2}\otimes \cdots \otimes u^{s,m} .
\]
In the following, we show that $\mF_{1:r,:,...,:} = \mH$
and there exist vectors $u^{1,1}, \ldots, u^{r,1} \in \bbC^{n_1}$
satisfying the equation \eqref{exist:u(s1)}.

%
By Theorem~\ref{decom_to_Mjk}, one can see that
the generating matrix $G$ for $\mF$
is also a generating matrix for $\mH$, so it holds that
\be \label{gpeqn:mF:mH}
\langle \phi[G,\tau] p,  \mF \rangle  =
\langle \phi[G,\tau] p,  \mH \rangle  =0, \quad
\mbox{for all} \,\, p \in \mathbb{M}_{\{1,j\}^c} .
\ee
Therefore, we have
\be  \label{GPeqn:H-hatF}
\langle  \phi[G,\tau] p, \mH - \mF \rangle  =0,  \quad
\mbox{for all} \,\, p \in \mathbb{M}_{\{1,j\}^c}.
\ee
By \eqref{proof}, one can see that
\begin{align}
(\mH - \mF)_{1:r,1,\ldots,1}=0. \label{first:zero}
\end{align}
In \eqref{GPeqn:H-hatF}, for each $\tau = (i, 2, k) \in J$
and $p = 1$, we can get
\[
(\mH - \mF )_{1:r,:,1,\ldots,1}=0.
\]
Similarly, for $\tau = (i, 2, k) \in J$ and $p=x_{3,j_3}$, we can get
\[
(\mH  - \mF )_{1:r,:,:,1,\ldots,1}=0.
\]
Continuing this, we can eventually get
$\mH  = \mF_{1:r,:,:,\ldots, :}$.
Since the matrix $(V^{(1)})_{1:r,:}$ is invertible, there exists a matrix
$W \in \mathbb{C}^{n_1\times r} $ such that
$V^{(1)} =  W (V^{(1)})_{1:r,:}.$
Observe that
\[
\mathcal{F} = W \times_1 \mF_{1:r,:,:,\ldots, :}= W \times_1 \mH.
\]
Let $u^{s,1}=W \cdot (c_sp_s)$ for $s=1,\ldots,r$.
Then the tensor decomposition (\ref{exist:u(s1)}) holds.
\end{proof}

\subsection{An algorithm for computing tensor decompositions}

%
%
Consider a tensor $\mF \in\bbC^{n_1 \times n_2 \times \cdots \times n_m}$
with a given rank $r$. Recall that the dimensions are ordered such that
$n_1 \ge n_2 \ge \cdots \ge n_m$.
We discuss how to compute the rank-$r$ tensor decomposition for $\mF$.
Recall $A[\mathcal{F},j]$, $b[\mathcal{F},j,k]$ as in \eqref{def:A_b}, for $j > 1$.
Note that $A[\mathcal{F},j]$ has the dimension $N_j \times r$, where
\be \label{Nj:notj}
N_j \, \coloneqq \,
\frac{n_2 \cdots n_m}{n_j}   .
\ee
If $r \le N_j$, then the matrices $A[\mathcal{F},j]$
have full column rank for generic cases.
For instance, $r \le N_3$ if $m=3$ and $r \le n_2$.
Since $N_2$ is the smallest,
we often use the matrices $A[\mathcal{F},j]$ for $j \ge 3$.
For convenience, denote the label set
\be \label{set:Upsi}
\Upsilon \, \coloneqq \, \{(j, k):
3 \le j \le m, 2 \le k \le n_j \}.
\ee

In the following, we consider the case that $r \le N_3$.
For each pair $(j,k) \in \Upsilon$, the linear system \eqref{linear_eq}
has a unique solution, for which we denote
\[
Y^{j,k} \, = \, M^{j,k}[G].
\]
For $j=2$, the equation \eqref{linear_eq} may not
have a unique solution if $r > N_2$.
In the following, we show how to get the tensor decomposition
without using the matrices $M^{2,k}[G]$.
By Theorem~\ref{decom_to_Mjk}, the matrices $Y^{j,k}$
are simultaneously diagonalizable, that is,
there is an invertible matrix $P \in \bbC^{r \times r}$
such that all products $P^{-1}  Y^{j,k}  P$
are diagonal for every $(j,k) \in \Upsilon$.
Suppose they are diagonalized as
\be \label{diag:eig:Yjk}
P^{-1}  Y^{j,k}  P \, = \,
\mbox{diag}[\lambda_{j,k,1},\lambda_{j,k,2},\ldots,\lambda_{j,k,r}]
\ee
with the eigenvalues $\lambda_{j,k,s}$. Write $P$ in the column form
\[
P \, = \, \bbm p_1 & \cdots & p_r \ebm.
\]
For each $s =1,\ldots,r$ and $j=3,\ldots, m$, let
\be \label{vsj:j3:m}
v^{s,j}  \, \coloneqq \,  (1, \lambda_{j,2, s}, \ldots, \lambda_{j,n_j,s} ).
\ee
Suppose $\mathcal{F}$ has the rank-$r$ decomposition
\[
\mathcal{F} \, = \,  \sum_{s=1}^r u^{s,1} \otimes ... \otimes u^{s,m}.
\]
Under the assumptions of Theorem~\ref{decom_to_Mjk},
the linear system \eqref{linear_eq}
has a unique solution  for each pair $(j,k) \in \Upsilon$.
For every $j \in \{3,...,m\}$, there exist scalars $c_{s,j},c_{s,1}$ such that
\[
 u^{s,j} \, = \, c_{s,j} v^{s,j}, \quad
 u^{s,1} \, = \, c_{s,1} p_s  .
\]
Then, we consider the sub-tensor equation
in the vector variables $y_1, \ldots, y_r \in \bbC^{n_2}$
\be \label{eq:ys:j=2}
\mF_{1:r,:,\ldots, :} \, = \, \sum_{s=1}^r p_s \otimes y_s \otimes
v^{s,3} \otimes \cdots  \otimes v^{s,m}.
\ee
There are $r n_2 \cdots n_m$ equations and $rn_2$ unknowns.
This overdetermined linear system has solutions such that
\[
y_s=c_{s,2} u^{s,2}, ~\mathrm{for~some}~c_{s,2} \in \mathbb{C}.
\]
After all $y_s$ are obtained,  we solve the linear equation in
$z_1, \ldots, z_r \in \bbC^{n_1-r}$
\be  \label{exist:v(s1)}
\mF_{r+1:n_1,:,\ldots, :} =\sum_{s=1}^r z_s \otimes y_s
\otimes v^{s,3} \otimes \cdots  \otimes v^{s,m} .
\ee
After all $y_s, z_s$ are obtained,
we choose the vectors ($s=1,\ldots,r$)
\[
v^{s,1} = \bbm p_s \\ z_s \ebm, \quad    v^{s,2} = y_s.
\]
Then we get the tensor decomposition
\be  \label{TDF:alg:vsj}
\mF  \, = \, \sum_{s=1}^r   v^{s,1} \otimes v^{s,2} \otimes \cdots  \otimes v^{s,m} .
\ee

Summarizing the above, we get the following algorithm
for computing tensor decompositions when $r\le n_1$ and $r \le N_3$.
Suppose the dimensions are ordered such that
$n_1 \ge n_2 \ge \cdots \ge n_m$.

\begin{alg} \label{alg:TD:r<=n2}
(Rank-$r$ tensor decomposition.)

\begin{itemize}

\item[Input:] A tensor $\mathcal{F}\in \bbC^{n_1  \times ... \times n_m}$
    with rank $r \leq \min (n_1, N_3)$.

\item[Step~1] For each pair $(j,k) \in \Upsilon$,
solve the matrix equation for the solution $Y^{j,k}$:
\be
 A[\mF,j] Y^{j,k}  \, = \, b[\mathcal{F},j,k].
\ee

\item[Step~2] Choose generic scalars $\xi_{j,k} $.
Then compute the eigenvalue decomposition $P^{-1} Y P = D$ for the matrix
\[
 Y \, \coloneqq \, \frac{1}{\sum\limits_{ (j,k) \in \Upsilon}  \xi_{j,k}}
 \sum\limits_{ (j,k) \in \Upsilon  }
 \xi_{j,k}  Y^{j,k} .
\]

\item[Step~3] For $s = 1, \ldots, r$ and $j \ge 3$,
let $v^{s,j}$ be the vectors as in \eqref{vsj:j3:m}.

\item[Step~4] Solve the linear system \eqref{eq:ys:j=2}
for vectors $y_1, \ldots, y_r$.

\item[Step~5] Solve the linear system \eqref{exist:v(s1)}
for vectors $z_1, \ldots, z_r$.

\item[Step~6] For each $s = 1, \ldots, r$,
let $v^{s,1} = \bbm p_s \\ z_s \ebm$ and $v^{s,2} = y_s$.

\item [Output:] The tensor rank-$r$ decomposition as in \eqref{TDF:alg:vsj}.

\end{itemize}
\end{alg}

The correctness of Algorithm~\ref{alg:TD:r<=n2} is justified as follows.

\begin{theorem}  \label{rank1:decom}
Suppose $n_1 \ge n_2 \ge \cdots \ge n_m$ and
$r \le \min (n_1, N_3)$ as in \eqref{Nj:notj}.
For a generic  tensor $\mF$ of rank-$r$,
Algorithm~\ref{alg:TD:r<=n2} produces a rank-$r$ tensor decomposition for $\mF$.
\end{theorem}
\begin{proof}
This can be implied by Theorem~\ref{onetoone}.
\end{proof}

\subsection{Tensor decompositions via reshaping}
\label{ssc:td:reshape}

A tensor $\mF \in \bbC^{n_1 \times \cdots \times n_m}$
can be reshaped as a cubic order tensor $\widehat{\mF}$ as in \eqref{outer_product_form}.
One can apply Algorithm~\ref{alg:TD:r<=n2} to compute the tensor decomposition
\eqref{outer_product_form} for $\widehat{\mF}$.
If the decomposing vectors $w^{s,1}, w^{s,2}, w^{s,3}$
can be reshaped to rank-$1$ tensors, then we can convert
\eqref{outer_product_form} to a tensor decomposition for $\mF$.
This is justified by Theorem~\ref{fla}, under some assumptions.
A benefit for doing this is that we may be able to compute tensor decompositions
for the case that \[ N_3 < r \le p_2, \]
with the dimension $p_2$ as in Theorem~\ref{fla}.
This leads to the following algorithm for computing tensor decompositions.

\begin{alg}  	\label{algorithm:flattening}
(Tensor decompositions via reshaping.)
Let $p_1, p_2, p_3$ be dimensions as in Theorem~\ref{fla}.

\begin{itemize}

\item[Input:] A tensor $\mF \in \bbC^{n_1 \times \cdots  \times n_m}$
with rank $r \le p_2$.

\item[Step~1] Reshape the tensor $\mF$ to
a cubic tensor $\widehat{\mF} \in \bbC^{p_1 \times p_2 \times p_3}$
as in \eqref{outer_product_form}.

\item[Step~2] Use Algorithm~\ref{alg:TD:r<=n2} to compute the
tensor decomposition
\be \label{TD:reshape}
\widehat{\mF} \, = \, \sum_{s=1}^r
w^{s,1}\otimes w^{s,2} \otimes w^{s,3} .
\ee

\item[Step~3] If all $w^{s,1}, w^{s,2}, w^{s,3}$ can be expressed as
outer products of rank-$1$ tensors as in \eqref{KRprod:vwz},
then output the tensor decomposition as in \eqref{mATD:usj}.
If one of  $w^{s,1}, w^{s,2}, w^{s,3}$ cannot be expressed
as in \eqref{KRprod:vwz}, then the reshaping does not produce
a tensor decomposition for $\mF$.

\item [Output:] A tensor decomposition for $\mF$ as in \eqref{mATD:usj}.

\end{itemize}

\end{alg}

For Algorithm~\ref{algorithm:flattening},
we have a similar conclusion like Theorem~\ref{rank1:decom}.
For cleanness of the paper, we do not repeat it here.

\section{low rank Tensor Approximations}
\label{sec:lowTA}

When a tensor $\mF \in \mathbb{C}^{n_1 \times ... \times n_m}$
has the rank bigger than $r$, the linear systems in
Algorithm~\ref{alg:TD:r<=n2} may not be consistent.
However, we can find linear least squares solutions for them.
This gives an algorithm for computing low rank tensor approximations.
Recall the label set $\Upsilon$ as in \eqref{set:Upsi}.
The following is the algorithm.
\begin{alg} \label{alg:LRTA}
(Rank-$r$ tensor approximation.)
\begin{itemize}

\item [Input:]
A tensor $\mathcal{F}\in \bbC^{n_1 \times n_2 \times ... \times n_m} $
and a rank $r \leq \min (n_1, N_3)$.

\item [Step~1]  For each pair $(j,k) \in \Upsilon$,
solve the linear least squares problem
\be 	\label{LS:Yjk:j>=3}
\min\limits_{ Y^{j,k} \in \bbC^{r \times r} }  \quad
\bigg\lVert A[\mathcal{F},j] (Y^{j,k} )^T - b[\mathcal{F},j,k] \bigg\rVert^2 .
\ee
Let $\hat{Y}^{j,k}$ be an optimizer.

\item [Step~2]
Choose generic scalars $\xi_{j,k}$
and let
\[
\hat{Y}[\xi] \, = \, \frac{1}{\sum\limits_{ (j,k) \in \Upsilon}  \xi_{j,k}}
\sum\limits_{ (j,k) \in \Upsilon } \xi_{j,k} \hat{Y}^{j,k} .
\]
Compute the eigenvalue decomposition
$\hat{P}^{-1} \hat{Y}[\xi] \hat{P} = \Lambda$ such that
$\hat{P} = \bbm \hat{p}_1 & \cdots & \hat{p}_r \ebm$
is invertible and $\Lambda$ is diagonal.

\item [Step~3]  For each pair $(j,k) \in \Upsilon$,
select the diagonal entries
\[
\mbox{diag}[\hat{\lambda}_{j,k,1}  \,\, \hat{\lambda}_{j,k,2}
\,\, \ldots \,\, \hat{\lambda}_{j,k,r}   ]  \, = \,
\mbox{diag} (\hat{P}^{-1} \hat{Y}^{j,k} \hat{P}).
\]
For each $s=1,\ldots,r$ and $j= 3, \ldots, m$, let
\[
\hat{v}^{s,j} = (1, \hat{\lambda}_{j,2,2}, \ldots, \hat{\lambda}_{j,n_j,s}  ).
\]

\item [Step~4] Let $(\hat{y}_1, \ldots, \hat{y}_r)$
be an optimizer for the following least squares:
\be  \label{LS:hatps}
\min\limits_{ (y_1, \ldots, y_r) }  \quad \bigg\lVert
\mF_{1:r, :, \ldots, :} - \sum_{s=1}^r
\hat{p}_s \otimes y_s \otimes \hat{v}^{s,3} \otimes \cdots \otimes  \hat{v}^{s,m}
\bigg\rVert^2.
\ee

\item [Step~5] Let $(\hat{z}_1, \ldots, \hat{z}_r)$
be an optimizer for the following least squares:
\be  \label{LS:hatzs}	
\min\limits_{ (z_1, \ldots, z_r) }  \quad \bigg\lVert
\mF_{r+1:n_1, :, \ldots, :} - \sum_{s=1}^r z_s \otimes \hat{y}_s
\otimes \hat{v}^{s,3} \otimes \cdots \otimes  \hat{v}^{s,m}
\bigg\rVert^2.
\ee

\item [Step~6]  Let $\hat{v}^{s,1} = \bbm \hat{p}_s \\ \hat{z}_s \ebm$
and $\hat{v}^{s,2} = \hat{y}_s$ for each
$s=1,\ldots, r$.

\item [Output:] The rank-$r$ approximation tensor
\be	\label{LS:Xgp}
\mathcal{X}^{gp} \, \coloneqq \,
\sum_{s=1}^r \hat{v}^{s,1} \otimes \hat{v}^{s,2} \otimes ... \otimes \hat{v}^{s,m}.
\ee

\end{itemize}

\end{alg}

If $\mathcal{F}$ is sufficiently close to a rank-$r$ tensor,
then $\mathcal{X}^{gp}$ is expected to be a good rank-$r$ approximation.
Mathematically, the tensor $\mathcal{X}^{gp}$ produced by Algorithm~\ref{alg:LRTA}
may not be the best rank-$r$ approximation.  However, in computational practice,
we can use \eqref{LS:Xgp} as a starting point
to solve the nonlinear least squares optimization
\be 	\label{least_squares_3}
\min\limits_{ (u^{s,1}, \ldots, u^{s,m}) } \bigg\lVert
\mF - \sum_{s=1}^r u^{s,1} \otimes u^{s,2} \otimes ... \otimes u^{s,m}
\bigg\rVert^2.
\ee
to improve the approximation quality.
Let $\mX^{opt}$ be the rank-$r$ approximation tensor
\be	\label{LS:Xopt}
\mathcal{X}^{opt} \, \coloneqq \,
\sum_{s=1}^r u^{s,1} \otimes u^{s,2} \otimes ... \otimes u^{s,m}
\ee
which is an optimizer to \reff{least_squares_3} obtained by
nonlinear optimization methods with
$\mX^{opt}$ as the initial point.

\subsection{Approximation error analysis}

Suppose the tensor $\mF$ has the best (or nearly best) rank-$r$ approximation
\be \label{X_bs}
\mX^{bs} \, \coloneqq \, \sum_{s=1}^{r}
(x^{s,1}) \otimes (x^{s,2}) \otimes ... \otimes (x^{s,m}).
\ee
Let $\mE$ be the tensor such that
\begin{align} 	\label{tensor_error}
\mF = \mathcal{X}^{bs} + \mE.
\end{align}
We analyze the approximation performance of $\mX^{gp}$ when the distance
$\eps = \Vert \mathcal{E}\Vert$ is small.
For a generating matrix $G$ and a generic $\xi = (\xi_{j,k})_{(j,k) \in  \Upsilon}$,
denote that
\be  	\label{xi_sum}
M[\xi,G] \, \coloneqq \, \frac{1}{\sum\limits_{ (j,k) \in \Upsilon}  \xi_{j,k}}
\sum\limits_{(j,k) \in \Upsilon}
\xi_{j,k} M^{j,k}[G].
\ee
Recall the $A[\mathcal{F},j]$, $b[\mathcal{F},j,k]$ as in (\ref{def:A_b}).
Note that
\begin{align}
\begin{split}
	A[\mathcal{F},j]=&~A[\mathcal{X}^{bs},j]+A[\mathcal{\mathcal{E}},j], \\
	b[\mathcal{F},j,k]=&~b[\mathcal{X}^{bs},j,k]+b[\mathcal{\mathcal{E}},j,k].
\end{split}
\end{align}
Suppose $\left(x^{s, j}\right)_{1} \neq 0$ for $j=2, \ldots, m $.

\begin{theorem}  \label{thm:aprox}
Let $\mathcal{X}^{gp}$ be produced by Algorithm~\ref{alg:LRTA}.
Let $\mathcal{F},\mathcal{X}^{bs},\mathcal{X}^{opt},\mathcal{E},x^{s,j},\xi_{j,k}$ be as above.
Assume the following conditions hold:
\begin{enumerate}

\item [(i)] The subvectors $(x^{1,1})_{1:r}, \ldots,
(x^{r,1})_{1:r}$ are linearly independent.

\item [(ii)] All matrices $A[ \mathcal{F},j ]$
and $A[ \mathcal{X}^{bs},j ]$ ($3 \leq j \leq m$) have full column rank.

\item [(iii)] The first entry
$\left(x^{s, j}\right)_{1} \neq 0$ for all $j=2, \ldots, m $.

\item [(iv)] The following scalars are pairwisely distinct
\be  \label{pairwise_distinct}
\sum_{ (j,k) \in \Upsilon  }   \xi_{j, k}(x^{1,j})_k, ... ,
\sum_{ (j,k) \in \Upsilon  } \xi_{j ,k}(x^{r,j})_k .
\ee
\end{enumerate}
If the distance $\epsilon = \lVert \mF - \mX^{bs} \rVert$
is sufficiently small, then
\be
\lVert \mathcal{X}^{bs}-\mathcal{X}^{gp} \rVert = O(\epsilon) \quad
\quad \and \quad
\lVert \mathcal{F}-\mathcal{X}^{gp} \rVert = O(\epsilon).
\ee
where the constants in the above $O(\cdot)$ only depend on $\mathcal{F}$ and $\xi$.
\end{theorem}
\begin{proof}
By conditions (i) and (iii) and by Theorem~\ref{decom_to_Mjk},
there exists a generating matrix $G^{bs}$ for $\mX^{bs}$ such that
\be
A[\mathcal{X}^{b s}, j] (M^{j,k}[G^{bs}])^T \, = \,  b[\mathcal{X}^{b s},j,k]
\ee
for all $j \in \{2,\ldots,m\}$ and $k \in \{2,\ldots,n_j\}$.
Note that $Y^{j,k}$ is the least squares solution to  \eqref{LS:Yjk:j>=3},
so for each $(j,k) \in \Upsilon$,
\[
Y^{j,k}=A[\mathcal{F}, j]^{\dagger} \cdot b[\mathcal{F},j,k], \quad
M^{j,k}[G^{bs}_0]=A[\mathcal{X}^{b s}, j]^{\dagger} \cdot b[\mathcal{X}^{b s}, j,k ].
\]
(The super script $^\dagger$ denotes the Pseudo-inverse of a matrix.)
By (\ref{tensor_error}), for $j=2,\ldots,m$, we have
\be  	    \label{purterbation:ABerror}
\baray{l}
\left\|A[\mathcal{F}, j]-A[\mathcal{X}^{b s}, j]\right\|_{F} \le
  \left\| \mF - \mX^{bs} \right\|  \le     \epsilon,  \\
  \left\|b[\mathcal{F}, j,k ]-b[\mathcal{X}^{b s}, j,k ]\right\|_{F} \le
    \left\| \mF - \mX^{bs} \right\|  \le      \epsilon.
\earay
\ee
Hence, by the condition (ii), if $\epsilon>0$ is small enough, we have
\be 	    \label{purterbation:m3kerror}
\left\| Y^{j, k}  - M^{j, k}[G^{bs}]\right\| =  O(\epsilon).
\ee
for all $(j,k) \in \Upsilon$.
This follows from perturbation analysis for linear least squares
(see \cite[Theorem~3.4]{Demmel}).

By (\ref{X_bs}) and Theorem \ref{decom_to_Mjk},
for $s=1, \ldots, r$ and $(j,k) \in \Upsilon$, it holds that
\[
M^{j, k}[G^{bs}] \left(x^{s, 1}\right)_{1:r}
\quad = \quad
\left(x^{s, j}\right)_{k} 	\left(x^{s, 1}\right)_{1:r}  .
\]
This means that each $\left(x^{s, 1}\right)_{1: r}$
is an eigenvector of $M^{j, k}[G^{bs}]$,
associated to the eigenvalue $\left(x^{s, j}\right)_{k}$,
for each $s=$ $1, \ldots, r$.
The matrices $M^{j, k}[G^{bs}]$ are simultaneously diagonalizable,
by the condition (i). So $M[\xi,G^{bs}]$ is also diagonalizable.
Note the eigenvalues of $M[\xi,G^{bs}]$ are the sums in (\ref{pairwise_distinct}).
They are distinct from each other, by the condition (iv).
When $\epsilon>0$ is small enough, $M[\xi,G^{bs}]$ also has distinct eigenvalues.
Write that
\[
Q = \bbm (x^{1,1})_{1:r}  & \cdots  & (x^{r,1})_{1:r} \ebm .
\]
Note that $Q^{-1} M[\xi,G^{bs}]Q = D$
is an eigenvalue decomposition.
Up to a scaling on $\hat{P}$ in algorithm \ref{alg:LRTA}, it holds that
\be	\label{schur_bound}
\| \hat{p}_s - x^{s,1} \|_{2}=O(\epsilon), \quad
\|D - \Lambda \|_{F}=O(\epsilon).
\ee
We refer to \cite{chatelin2012}
for the perturbation bounds in \eqref{schur_bound}.
The constants in the above $O(\cdot)$ eventually only depend on $\mF,\xi$.

Note that $(\hat{y}_s,\ldots, \hat{y}_r)$ is the least squares
solution to \eqref{LS:hatps} and
\be
\mX^{bs}_{1:r, :, \ldots, :} = \sum_{s=1}^r
x^{s,1} \otimes x^{s,2} \otimes x^{s,3} \otimes \cdots \otimes  x^{s,m} .
\ee
Due to perturbation analysis of linear least squares, we also have
\be	\label{bound:ys-us2}
\| \hat{y}_s - x^{s,2} \|_{2}=O(\epsilon)  .
\ee
Note that the subvectors $(x^{s,1})_{r+1:n_1}$
satisfy the equation
\be
\mX^{bs}_{r+1:n_1, :, \ldots, :} = \sum_{s=1}^r
(x^{s,1})_{r+1:n_1} \otimes x^{s,2}  \otimes \cdots \otimes  x^{s,m} .
\ee
Recall that $(\hat{z}_1, \ldots, \hat{z}_r)$
is the least squares solution to \eqref{LS:hatzs}.
Due to perturbation analysis of linear least squares,
we further have the error bound
\be
\| (x^{s, 1})_{r+1:n_1} - \hat{z}_s  \|_{2}= O (\epsilon) .
\ee
Summarizing the above, we eventually get
$\|\mX^{g p}-\mX^{bs} \|=O(\epsilon)$, so
\[
\left\|\mathcal{F}-\mathcal{X}^{g p}\right\| \leq
\left\|\mathcal{F}-\mathcal{X}^{b s}\right\| +
\left\|\mathcal{X}^{b s}-\mathcal{X}^{g p}\right\| = O(\epsilon) .
\]
The constant for the above $O(\cdot)$ eventually only
depends on $\mF$, $\xi$.
\end{proof}

\subsection{Reshaping for low rank approximations}

Similar to tensor decompositions, the reshaping trick as in
Section~\ref{ssc:td:reshape} can also be
used for computing low rank tensor approximations.
For $m>3$, a tensor $\mF \in \bbC^{n_1 \times n_2 \times \cdots \times n_m}$
can be reshaped as a cubic tensor $\widehat{\mF} \in \bbC^{p_1 \times p_2 \times p_3}$
as in \eqref{outer_product_form}.
Similarly, Algorithm~\ref{alg:LRTA} can be used to compute
low rank tensor approximations.
Suppose the computed rank-$r$ approximating tensor
for $\widehat{\mF}$ is
\be  	\label{widehatX:gp}
\widehat{\mX}^{gp}:=\sum_{s=1}^{r}
\hat{w}^{s,1} \otimes \hat{w}^{s,2} \otimes \hat{w}^{s,3}.
\ee
Typically, the decomposing vectors
$\hat{w}^{s,1}, \hat{w}^{s,2},\hat{w}^{s,3}$
may not be reshaped to rank-$1$ tensors.
Suppose the reshaping is such that
$I_1 \cup I_2 \cup I_3 =\{1,2,...,m\}$
is a union of disjoint label sets and the reshaped dimensions are
\[
p_1 = \prod_{i \in I_1} n_i, \quad
p_2 = \prod_{i \in I_2} n_i, \quad
p_3 = \prod_{i \in I_3} n_i .
\]
Let $m_i = |I_i|$ for $i=1,2,3$. By the reshaping, the vectors
$\hat{w}^{s,i}$ can be reshaped back to a tensor
$\hat{W}^{s,i}$ of order $m_i$, for each $i=1,2,3$.
If $m_i = 1$, $\hat{W}^{s,i}$ is a vector.
If $m_i = 2$, we can find a best rank-$1$
matrix approximation for $\hat{W}^{s,i}$.
If $m_i \ge 3$, we can apply Algorithm~\ref{alg:LRTA} with $r=1$
to get a rank-$1$ approximation for $\hat{W}^{s,i}$.
In application, we are mostly interested in reshaping such that
all $m_i \le 2$. Finally, this produces
a rank-$r$ approximation for $\mF$.

The following is a low rank tensor approximation
algorithm via reshaping tensors.

\begin{alg}   \label{alg:LRTA:reshape}
(low rank tensor approximations via reshaping.)

\begin{itemize}

\item [Input] A tensor $\mF \in \bbC^{n_1 \times n_2 \times ... \times n_m}$ and a rank $r$.

\item [Step~1] Reshape $\mF$ to a cubic order tensor
$\widehat{\mF} \in \bbC^{p_1 \times p_2 \times p_3}$.

\item [Step~2] Use Algorithm~\ref{alg:LRTA} to compute
a rank-$r$ approximating tensor
$\widehat{\mX}^{gp}$ as in \eqref{widehatX:gp} for $\widehat{\mF}$.

\item [Step~3] For each $i=1,2,3$, reshape each vector $\hat{w}^{s,i}$
back to a tensor $\widehat{W}^{s,i}$ of order $m_i$ as above.

\item [Step~4] For each $i=1,2,3$, compute a rank-$1$ approximating
tensor $\widehat{X}^{s,i}$ for $\widehat{W}^{s,i}$ of order $m_i$ as above.

\item [Output]  Reshape the sum
$\sum\limits_{s=1}^r  \widehat{X}^{s,1} \otimes  \widehat{X}^{s,2} \otimes \widehat{X}^{s,3}$
to a tensor in $\bbC^{n_1 \times n_2 \times ... \times n_m}$,
which is a rank-$r$ approximation for $\mF$.

\end{itemize}

\end{alg}

We can do a similar approximation analysis for Algorithm~\ref{alg:LRTA:reshape}
as for Theorem~\ref{thm:aprox}.
For cleanness of the paper, we do not repeat that.

\section{Numerical Experiments}
\label{sec:num}

In this section, we apply Algorithms~\ref{alg:TD:r<=n2} and \ref{alg:LRTA}
to compute tensor decompositions and low rank tensor approximations.
We implement these algorithms in {\tt MATLAB} 2020b
on a workstation with Ubuntu 20.04.2 LTS,
Intel® Xeon(R) Gold 6248R CPU @ 3.00GHz and memory 1TB.
For computing low rank tensor approximations,
we use the function $cpd\_nls$ provided in
{\tt Tensorlab~3.0} \cite{7869679} to solve the nonlinear least squares optimization
\reff{least_squares_3}. The $\mathcal{X}^{gp}$ denotes the approximating tensor
returned by Algorithm~\ref{alg:LRTA} and
$\mathcal{X}^{opt}$ denotes the approximating tensor
obtained by solving \reff{least_squares_3},
with $\mathcal{X}^{gp}$ as the initial point.
In our numerical experiments,  if the rank $r$ is unknown,
we use the most square flattening matrix to estimate $r$
as in \eqref{ssc:flat:square} and Lemma \ref{cat_mat_lemma}.

\begin{exmp}
Consider the tensor $\mathcal{F}\in \mathbb{C}^{4 \times 4 \times 3}$
whose slices $\mathcal{F}_{:,:,1}, \mathcal{F}_{:,:,2}, \mathcal{F}_{:,:,3}$
are respectively
\begin{gather*}
\begin{bmatrix}
\begin{matrix*}[r]
27 & 25 & 35 & 42\\
48 & 68 & 80 & 80\\
26 & 24 & 34 & 40\\
33 & 41 & 49 & 66\\
\end{matrix*} & \rvline &
\begin{matrix*}[r]
44 & 32 & 52  & 56 \\
68 & 76 & 100 & 96 \\
42 & 30 & 50  & 52 \\
46 & 46 & 62  & 76 \\
\end{matrix*} &\rvline &
\begin{matrix*}[r]
42 & 26 & 48 & 45 \\
64 & 60 & 88 & 76 \\
47 & 27 & 53 & 47 \\
45 & 37 & 57 & 60 \\
\end{matrix*} \\
\end{bmatrix}.
\end{gather*}
By Lemma~\ref{cat_mat_lemma}, the estimated rank is $r=4$.
Applying Algorithm \ref{alg:TD:r<=n2} with $r=4$,
we get the rank-$4$ decomposition $\mF = U^{(1)} \circ U^{(2)} \circ U^{(3)}$,
with
\[
U^{(1)}=
\begin{bmatrix*}[r]
8 & 6  & 4  & 9  \\
8 & 12 & 16 & 12 \\
4 & 6  & 4  & 12 \\
4 & 12 & 8  & 9  \\
\end{bmatrix*},
~U^{(2)}=
\begin{bmatrix*}[r]
1 & 1  & 1  & 1  \\
\frac{1}{2} & 1  & 3  & \frac{1}{3}  \\
1 & 1  & 3  & 1  \\
1 & 4  & 1  & \frac{2}{3}  \\
\end{bmatrix*},
~U^{(3)}=
\begin{bmatrix*}[r]
1 & 1  & 1  & 1  \\
2 & 1  & 1  & 2  \\
1 & \frac{2}{3}  & \frac{3}{4}  & 3  \\
\end{bmatrix*}.
\]
\end{exmp}

\begin{exmp}
Consider the tensor in $\mathbb{C}^{5 \times 4 \times 3 \times 3}$
\begin{gather*} \mathcal{F}=
V^{(1)} \circ V^{(2)} \circ V^{(3)} \circ V^{(4)},
\end{gather*}
where the matrices $V^{(i)}$ are
\begin{gather*}
V^{(1)}=\begin{bmatrix*}[r]
	10 & 5 &-9 & -5 & 7\\
	8 & 6 & -3 & -9 & 7\\
	-9 & -1 & 7 & -3 & -1\\
	9 & -7 & -8 & 8 & -5\\
	-1 & 10 & 7 & -3 & 10 \\
	\end{bmatrix*}, \quad
V^{(2)}= \begin{bmatrix*}[r]
	-1 & 9 & -8 & 8 & 2\\
	0 & -1 & -4 & 6 & 8\\
	7 & -7 & -2 & 2 & 10\\
	2 & 10 & -3 & -1 & -3\\
	\end{bmatrix*},\\
V^{(3)}=
  \begin{bmatrix*}[r]
	5 & 2 & -2 & -7 & 3\\
	9 & -3 & -7 & 7 & -2\\
	0 & -10 & 10 & 6 & 10\\
	\end{bmatrix*}, \quad
V^{(4)}= \begin{bmatrix*}[r]
	8 & 2 & -7 & 10 & -5 \\
	4 & -8 & 4 & -6 & -10\\
	5 & 0 & 7 & -1 & -2\\
	\end{bmatrix*}.
\end{gather*}
By Lemma~\ref{cat_mat_lemma}, the estimated rank $r=5$.
Applying Algorithm \ref{alg:TD:r<=n2} with $r=5$,
we get the rank-$5$ tensor decomposition
$\mF = U^{(1)} \circ U^{(2)} \circ U^{(3)} \circ U^{(4)}$,
where the computed matrices $U^{(i)}$ are
\[
U^{(1)}=\begin{bmatrix*}[r]
-400 & 180     & 1008    & 2800   & -210    \\
-320  & 216     & 336     & 5040   & -210    \\
360   & -36     & -784    & 1680   & 30      \\
-360  & -252    & 896     & -4480  & 150     \\
40    & 360     & -784    & 1680   & -300    \\
\end{bmatrix*}, \quad
U^{(2)}=\begin{bmatrix*}[r]
1     & 1       & 1       & 1      & 1       \\
0     & -\frac{1}{9} & \frac{1}{2}    & \frac{3}{4}   & 4       \\
-7    & -\frac{7}{9} & \frac{1}{4}    & \frac{1}{4}   & 5       \\
-2    & \frac{10}{9}  & \frac{3}{8}   & -\frac{1}{8} & -\frac{3}{2}    \\
\end{bmatrix*},
\]
\[
U^{(3)}=\begin{bmatrix*}[r]
1     & 1       & 1       & 1      & 1       \\
\frac{9}{5}   & -\frac{3}{2}    & \frac{7}{2}     & -1     & -\frac{2}{3} \\
0     & -5      & -5      & -\frac{6}{7} & \frac{10}{3} \\
\end{bmatrix*}, \quad
U^{(4)}=\begin{bmatrix*}[r]
1     & 1       & 1       & 1      & 1       \\
\frac{1}{2}   & -4      & -\frac{4}{7} & -\frac{3}{5}   & 2       \\
\frac{5}{8} & 0       & -1      & -\frac{1}{10}   & \frac{2}{5}     \\
\end{bmatrix*}.
\]
\end{exmp}

\begin{exmp}
Consider the tensor $\mathcal{F} \in \mathbb{C}^{5\times 5 \times 4}$ such that
\begin{align*}
	\mathcal{F}_{i_1,i_2,i_3}=i_1+\frac{i_2}{2}+\frac{i_3}{3}+\sqrt{i_1^2+i_2^2+i_3^2}
\end{align*}
for  all $i_1,i_2,i_3$ in the corresponding range.
The $5$ biggest singular values of
the flattening matrix $\mbox{Flat}(\mathcal{F})$ are
\begin{align*}
   109.7393,~~
    5.2500,~~
    0.1068,~~
    8.325 \times 10^{-3},~~
    3.401 \times 10^{-4}.
\end{align*}
Applying Algorithm \ref{alg:LRTA} with rank $r=2,3,4,5$,
we get the approximation errors
\begin{center}
\begin{tabular}{ | c|c | c| c| c | } \hline
 $r$ & 2 & 3& 4& 5 \\
	\hline
	$\Vert \mathcal{F}-\mathcal{X}^{gp} \Vert$ & $5.1237\times 10^{-1}$ &
	$6.8647\times 10^{-2}$ & $1.0558\times 10^{-2}$ & $9.9449\times 10^{-3}$ \\
	\hline
	$\Vert \mathcal{F}-\mathcal{X}^{opt} \Vert$ & $1.5410\times 10^{-1}$ &
	$1.3754\times 10^{-2}$ & $2.6625\times 10^{-3}$ & $4.9002\times 10^{-4}$ \\
	\hline
\end{tabular}
\end{center}
For the case $r=3$,
the computed approximating tensor by Algorithm~\ref{alg:LRTA}
and by solving \eqref{least_squares_3} is
$U^{(1)} \circ U^{(2)} \circ U^{(3)}$, with
\begin{align*}
U^{(1)}=&
\begin{bmatrix*}[r]
   -0.4973 &  -7.6813 &  11.7465  \\
   -0.2525 &  -6.9651 &  12.4970  \\
   -0.0872 &  -6.0497 &  13.2858  \\
   -0.0132 &  -5.0521 &  14.1423  \\
   -0.0010 &  -4.0469 &  15.0771  \\
\end{bmatrix*},\quad
U^{(2)}= \begin{bmatrix*}[r]
    1.0000 &   1.0000 &   1.0000  \\
    0.5058 &   0.9211 &   1.0306  \\
    0.1713 &   0.8167 &   1.0649  \\
    0.0262 &   0.7003 &   1.1042  \\
    0.0136 &   0.5807 &   1.1490  \\
\end{bmatrix*}, \\
U^{(3)}=& \begin{bmatrix*}[r]
    1.0000 &   1.0000 &   1.0000  \\
    0.5075 &   0.9289 &   1.0216  \\
    0.1756 &   0.8323 &   1.0469  \\
    0.0399 &   0.7231 &   1.0771  \\
\end{bmatrix*}.
\end{align*}
\end{exmp}

\begin{exmp}
Consider the tensor $\mathcal{F} \in \mathbb{C}^{6\times 6 \times 6 \times 5 \times 4}$
such that
\[
\mathcal{F}_{i_1,i_2,i_3,i_4,i_5}= \
arctan(i_1 + 2i_2 + 3i_3 + 4i_4 + 5i_5),
\]
for all $i_1,i_2,i_3,i_4,i_5$ in the corresponding range.
The $5$ biggest singular values of the flattening matrix $\mbox{Flat}(\mathcal{F})$ are
\begin{align*}
101.71,~~~ 7.7529\times 10^{-2},~~~ 2.2870\times 10^{-3} ,
~~~7.2294\times 10^{-5},~~~2.0633\times 10^{-6}.
\end{align*}
Applying Algorithm \ref{alg:LRTA} with rank $r=2,3,4,5$,
we get the approximation errors as follows:
\begin{center}
\begin{tabular}{ | c|c | c| c| c|c|} \hline
 $r$  & 2 & 3 & 4 & 5\\ \hline
$\Vert \mathcal{F}-\mathcal{X}^{gp} \Vert$ & $9.8148 \times 10^{-3}$ &
$3.1987 \times 10^{-3}$ & $5.7945\times 10^{-3}$ & $1.0121 \times 10^{-5}$ \\ \hline
$\Vert \mathcal{F}-\mathcal{X}^{opt} \Vert$ &
$5.3111 \times 10^{-3}$ & $2.2623\times 10^{-4}$ & $3.0889 \times 10^{-5}$ &
$1.7523 \times 10^{-6}$\\ \hline
\end{tabular}
\end{center}
For the case $r=3$,
the computed approximating tensor by Algorithm~\ref{alg:LRTA}
and by solving \eqref{least_squares_3} is
$U^{(1)} \circ U^{(2)} \circ U^{(3)} \circ U^{(4)} \circ U^{(5)}$, with
\begin{align*}
 U^{(1)}=&\begin{bmatrix*}[r]
   -0.0134 &  -0.0347 &   1.5524  \\
   -0.0112 &  -0.0329 &   1.5525  \\
   -0.0094 &  -0.0312 &   1.5526  \\
   -0.0079 &  -0.0295 &   1.5527  \\
   -0.0066 &  -0.0280 &   1.5528  \\
   -0.0056 &  -0.0265 &   1.5529  \\
    \end{bmatrix*}, \quad
U^{(2)}=\begin{bmatrix*}[r]
    1.0000 &   1.0000 &   1.0000  \\
    0.7011 &   0.8992 &   1.0001  \\
    0.4939 &   0.8080 &   1.0003  \\
    0.3485 &   0.7260 &   1.0004  \\
    0.2459 &   0.6523 &   1.0006  \\
    0.1734 &   0.5861 &   1.0007  \\
    \end{bmatrix*},
\end{align*}
\begin{align*}
U^{(3)}=&\begin{bmatrix*}[r]
    1.0000 &   1.0000 &   1.0000  \\
    0.5886 &   0.8523 &   1.0002  \\
    0.3490 &   0.7258 &   1.0004  \\
    0.2064 &   0.6183 &   1.0006  \\
    0.1214 &   0.5269 &   1.0008  \\
    0.0715 &   0.4489 &   1.0011  \\
    \end{bmatrix*}, \quad
U^{(4)}=\begin{bmatrix*}[r]
    1.0000 &   1.0000 &   1.0000  \\
    0.4949 &   0.8078 &   1.0003  \\
    0.2463 &   0.6521 &   1.0006  \\
    0.1211 &   0.5269 &   1.0008  \\
    0.0596 &   0.4256 &   1.0011  \\
    \end{bmatrix*}, \\
U^{(5)}=&\begin{bmatrix*}[r]
    1.0000 &   1.0000 &   1.0000  \\
    0.4161 &   0.7656 &   1.0003  \\
    0.1730 &   0.5862 &   1.0007  \\
    0.0711 &   0.4489 &   1.0011  \\
\end{bmatrix*}.
\end{align*}
\end{exmp}

\begin{exmp}
As in Theorem~\ref{thm:aprox}, we have shown that if
the tensor to be approximated is sufficiently close to
a rank-$r$ tensor, then the computed rank-$r$ approximation
$\mathcal{X}^{gp}$ is quasi-optimal.
It can be further improved to a better approximation $\mathcal{X}^{opt}$
by solving the nonlinear optimization \reff{least_squares_3}.
In this example, we explore the numerical performance of
Algorithms~\ref{alg:LRTA} and \ref{alg:LRTA:reshape}
for computing low rank tensor approximations.
For the given dimensions $n_{1}, \ldots, n_{m}$,
we generate the tensor
\[
\mathcal{R} \, = \, \sum_{s=1}^{r}
 u^{s, 1} \otimes u^{s, 2} \otimes \cdots \otimes u^{s, m},
\]
where each $u^{s, j} \in \mathbb{C}^{n_{j}}$
is a complex vector whose real and imaginary parts
are generated randomly, obeying the Gaussian distribution.
We perturb $\mathcal{R}$ by another tensor $\mathcal{E}$,
whose entries are also generated with the Gaussian distribution.
We scale the perturbing tensor $\mathcal{E}$ to have a desired norm $\epsilon$.
The tensor $\mF$ is then generated as
\[
\mathcal{F} \, = \,  \mathcal{R}+\mathcal{E}.
\]
We choose $\epsilon$ to be one of $10^{-2}, 10^{-4}, 10^{-6}$,
and use the relative errors
\[
\text {\(\rho\)\_gp}= \frac{\left\|\mathcal{F}-\mathcal{X}^{gp}\right\|}
{\|\mathcal{E}\|},
\quad
\text {\(\rho\)\_opt }= \frac{\left\|\mathcal{F}-\mathcal{X}^{opt}\right\| }
{\|\mathcal{E}\|}
\]
to measure the approximation quality of
$\mathcal{X}^{\text {gp }}$, $\mathcal{X}^{\text {opt }}$ respectively.
For each case of $(n_1, \ldots, n_m), r$ and $\epsilon$,
we generate $10$ random instances of $\mathcal{R}, \mathcal{F}, \mathcal{E}$.
For the case $(n_1, \ldots, n_m) = (20,20,20,20,10)$,
Algorithm~\ref{alg:LRTA:reshape} is used to compute $\mathcal{X}^{gp}$.
All other cases are solved by Algorithm~\ref{alg:LRTA}.
The computational results are reported in Tables~\ref{Table:LRTA}.
For each case of $(n_{1}, \ldots, n_{m})$ and $r$,
we also list the median of above relative errors
and the average CPU time (in seconds).
The \text{t\_gp} and \text{t\_opt} denote the average CPU time (in seconds)
for Algorithms~\ref{alg:LRTA}/\ref{alg:LRTA:reshape} and for
solving \reff{least_squares_3} respectively.

\begin{table}
\caption{Computational performance of Algorithms~\ref{alg:LRTA} and \ref{alg:LRTA:reshape} and of nonlinear optimization \reff{least_squares_3}.}
\label{Table:LRTA}
	\resizebox{\textwidth}{!}{
\begin{tabular}{|cccccc|cccccc|}\hline
	\multicolumn{1}{|c|}{$r$}                   & \multicolumn{1}{c|}{$\epsilon$} & \multicolumn{1}{c|}{\(\rho\)\_{gp}}  & \multicolumn{1}{c|}{t\_{gp}} & \multicolumn{1}{c|}{\(\rho\)\_{opt}} & t\_{opt}   & \multicolumn{1}{c|}{$r$}                   & \multicolumn{1}{c|}{$\epsilon$} & \multicolumn{1}{c|}{\(\rho\)\_{gp}}  & \multicolumn{1}{c|}{t\_{gp}} & \multicolumn{1}{c|}{\(\rho\)\_{opt}} & \multicolumn{1}{c|}{t\_{opt}}    \\ \hline
	
	\multicolumn{6}{|c|}{($n_1$,$n_2$,$n_3$)=(50,50,50)}  &
	\multicolumn{6}{|c|}{($n_1$,$n_2$,$n_3$)=(60, 50, 40)} \\ \hline	
	
	\multicolumn{1}{|c|}{\multirow{3}{*}{10}} &  \multicolumn{1}{c|}{$10^{-2}$} & \multicolumn{1}{c|}{1.63}& \multicolumn{1}{c|}{0.08}& \multicolumn{1}{c|}{0.99}& \multicolumn{1}{c|}{1.57}& \multicolumn{1}{|c|}{\multirow{3}{*}{15}} &  \multicolumn{1}{c|}{$10^{-2}$}  & \multicolumn{1}{c|}{17.49}& \multicolumn{1}{c|}{0.19}& \multicolumn{1}{c|}{0.99}& \multicolumn{1}{c|}{2.17}\\ \cline{2-6}\cline{8-12}
	\multicolumn{1}{|c|}{} &  \multicolumn{1}{c|}{$10^{-4}$} & \multicolumn{1}{c|}{6.32}& \multicolumn{1}{c|}{0.10}& \multicolumn{1}{c|}{0.99}& \multicolumn{1}{c|}{1.16}&  \multicolumn{1}{|c|}{} &  \multicolumn{1}{c|}{$10^{-4}$}  & \multicolumn{1}{c|}{10.80}& \multicolumn{1}{c|}{0.15}& \multicolumn{1}{c|}{0.99}& \multicolumn{1}{c|}{1.36}\\ \cline{2-6}\cline{8-12}
	\multicolumn{1}{|c|}{} &  \multicolumn{1}{c|}{$10^{-6}$} & \multicolumn{1}{c|}{3.84}& \multicolumn{1}{c|}{0.09}& \multicolumn{1}{c|}{0.99}& \multicolumn{1}{c|}{0.83}&  \multicolumn{1}{|c|}{} &  \multicolumn{1}{c|}{$10^{-6}$}  & \multicolumn{1}{c|}{5.16}& \multicolumn{1}{c|}{0.20}& \multicolumn{1}{c|}{0.99}& \multicolumn{1}{c|}{1.10}\\ \hline
	\multicolumn{1}{|c|}{\multirow{3}{*}{20}} &  \multicolumn{1}{c|}{$10^{-2}$} & \multicolumn{1}{c|}{25.83}& \multicolumn{1}{c|}{0.29}& \multicolumn{1}{c|}{0.99}& \multicolumn{1}{c|}{2.99}& \multicolumn{1}{|c|}{\multirow{3}{*}{30}} &  \multicolumn{1}{c|}{$10^{-2}$}  & \multicolumn{1}{c|}{28.70}& \multicolumn{1}{c|}{0.40}& \multicolumn{1}{c|}{0.98}& \multicolumn{1}{c|}{6.95}\\ \cline{2-6}\cline{8-12}
	\multicolumn{1}{|c|}{} &  \multicolumn{1}{c|}{$10^{-4}$} & \multicolumn{1}{c|}{5.41}& \multicolumn{1}{c|}{0.28}& \multicolumn{1}{c|}{0.99}& \multicolumn{1}{c|}{1.99}&  \multicolumn{1}{|c|}{} &  \multicolumn{1}{c|}{$10^{-4}$}  & \multicolumn{1}{c|}{15.77}& \multicolumn{1}{c|}{0.37}& \multicolumn{1}{c|}{0.98}& \multicolumn{1}{c|}{3.61}\\ \cline{2-6}\cline{8-12}
	\multicolumn{1}{|c|}{} &  \multicolumn{1}{c|}{$10^{-6}$} & \multicolumn{1}{c|}{30.41}& \multicolumn{1}{c|}{0.29}& \multicolumn{1}{c|}{0.99}& \multicolumn{1}{c|}{1.49}&  \multicolumn{1}{|c|}{} &  \multicolumn{1}{c|}{$10^{-6}$}  & \multicolumn{1}{c|}{50.96}& \multicolumn{1}{c|}{0.37}& \multicolumn{1}{c|}{0.98}& \multicolumn{1}{c|}{2.27}\\ \hline
	\multicolumn{1}{|c|}{\multirow{3}{*}{30}} &  \multicolumn{1}{c|}{$10^{-2}$} & \multicolumn{1}{c|}{27.91}& \multicolumn{1}{c|}{0.50}& \multicolumn{1}{c|}{0.98}& \multicolumn{1}{c|}{7.08}& \multicolumn{1}{|c|}{\multirow{3}{*}{45}} &  \multicolumn{1}{c|}{$10^{-2}$}  & \multicolumn{1}{c|}{35.48}& \multicolumn{1}{c|}{0.61}& \multicolumn{1}{c|}{0.97}& \multicolumn{1}{c|}{25.73}\\ \cline{2-6}\cline{8-12}
	\multicolumn{1}{|c|}{} &  \multicolumn{1}{c|}{$10^{-4}$} & \multicolumn{1}{c|}{213.82}& \multicolumn{1}{c|}{0.43}& \multicolumn{1}{c|}{0.98}& \multicolumn{1}{c|}{3.73}&  \multicolumn{1}{|c|}{} &  \multicolumn{1}{c|}{$10^{-4}$}  & \multicolumn{1}{c|}{35.03}& \multicolumn{1}{c|}{0.63}& \multicolumn{1}{c|}{0.97}& \multicolumn{1}{c|}{8.08}\\ \cline{2-6}\cline{8-12}
	\multicolumn{1}{|c|}{} &  \multicolumn{1}{c|}{$10^{-6}$} & \multicolumn{1}{c|}{17.97}& \multicolumn{1}{c|}{0.47}& \multicolumn{1}{c|}{0.98}& \multicolumn{1}{c|}{2.20}&  \multicolumn{1}{|c|}{} &  \multicolumn{1}{c|}{$10^{-6}$}  & \multicolumn{1}{c|}{34.67}& \multicolumn{1}{c|}{0.61}& \multicolumn{1}{c|}{0.97}& \multicolumn{1}{c|}{5.69}\\ \hline
	
	\multicolumn{6}{|c|}{($n_1$,$n_2$,$n_3$)=(100,100,100)}  &
	\multicolumn{6}{|c|}{($n_1$,$n_2$,$n_3$)=(150,150,150)} \\ \hline

	\multicolumn{1}{|c|}{\multirow{3}{*}{20}} &  \multicolumn{1}{c|}{$10^{-2}$} & \multicolumn{1}{c|}{11.21}& \multicolumn{1}{c|}{0.86}& \multicolumn{1}{c|}{1.00}& \multicolumn{1}{c|}{6.36}& \multicolumn{1}{|c|}{\multirow{3}{*}{30}} &  \multicolumn{1}{c|}{$10^{-2}$}  & \multicolumn{1}{c|}{8.59}& \multicolumn{1}{c|}{2.92}& \multicolumn{1}{c|}{1.00}& \multicolumn{1}{c|}{17.17}\\ \cline{2-6}\cline{8-12}
	\multicolumn{1}{|c|}{} &  \multicolumn{1}{c|}{$10^{-4}$} & \multicolumn{1}{c|}{3.48}& \multicolumn{1}{c|}{0.85}& \multicolumn{1}{c|}{1.00}& \multicolumn{1}{c|}{4.24}&  \multicolumn{1}{|c|}{} &  \multicolumn{1}{c|}{$10^{-4}$}  & \multicolumn{1}{c|}{3.18}& \multicolumn{1}{c|}{3.05}& \multicolumn{1}{c|}{1.00}& \multicolumn{1}{c|}{11.20}\\ \cline{2-6}\cline{8-12}
	\multicolumn{1}{|c|}{} &  \multicolumn{1}{c|}{$10^{-6}$} & \multicolumn{1}{c|}{3.88}& \multicolumn{1}{c|}{0.83}& \multicolumn{1}{c|}{1.00}& \multicolumn{1}{c|}{3.20}&  \multicolumn{1}{|c|}{} &  \multicolumn{1}{c|}{$10^{-6}$}  & \multicolumn{1}{c|}{4.24}& \multicolumn{1}{c|}{3.42}& \multicolumn{1}{c|}{1.00}& \multicolumn{1}{c|}{11.75}\\ \hline
	\multicolumn{1}{|c|}{\multirow{3}{*}{40}} &  \multicolumn{1}{c|}{$10^{-2}$} & \multicolumn{1}{c|}{24.17}& \multicolumn{1}{c|}{1.76}& \multicolumn{1}{c|}{0.99}& \multicolumn{1}{c|}{17.80}& \multicolumn{1}{|c|}{\multirow{3}{*}{60}} &  \multicolumn{1}{c|}{$10^{-2}$}  & \multicolumn{1}{c|}{49.80}& \multicolumn{1}{c|}{6.04}& \multicolumn{1}{c|}{1.00}& \multicolumn{1}{c|}{87.31}\\ \cline{2-6}\cline{8-12}
	\multicolumn{1}{|c|}{} &  \multicolumn{1}{c|}{$10^{-4}$} & \multicolumn{1}{c|}{11.60}& \multicolumn{1}{c|}{1.65}& \multicolumn{1}{c|}{0.99}& \multicolumn{1}{c|}{11.02}&  \multicolumn{1}{|c|}{} &  \multicolumn{1}{c|}{$10^{-4}$}  & \multicolumn{1}{c|}{13.77}& \multicolumn{1}{c|}{5.89}& \multicolumn{1}{c|}{1.00}& \multicolumn{1}{c|}{24.96}\\ \cline{2-6}\cline{8-12}
	\multicolumn{1}{|c|}{} &  \multicolumn{1}{c|}{$10^{-6}$} & \multicolumn{1}{c|}{11.09}& \multicolumn{1}{c|}{1.61}& \multicolumn{1}{c|}{0.99}& \multicolumn{1}{c|}{7.97}&  \multicolumn{1}{|c|}{} &  \multicolumn{1}{c|}{$10^{-6}$}  & \multicolumn{1}{c|}{17.49}& \multicolumn{1}{c|}{6.07}& \multicolumn{1}{c|}{1.00}& \multicolumn{1}{c|}{18.81}\\ \hline
	\multicolumn{1}{|c|}{\multirow{3}{*}{60}} &  \multicolumn{1}{c|}{$10^{-2}$} & \multicolumn{1}{c|}{18.71}& \multicolumn{1}{c|}{3.40}& \multicolumn{1}{c|}{0.99}& \multicolumn{1}{c|}{28.16}& \multicolumn{1}{|c|}{\multirow{3}{*}{90}} &  \multicolumn{1}{c|}{$10^{-2}$}  & \multicolumn{1}{c|}{29.44}& \multicolumn{1}{c|}{10.64}& \multicolumn{1}{c|}{0.99}& \multicolumn{1}{c|}{98.78}\\ \cline{2-6}\cline{8-12}
	\multicolumn{1}{|c|}{} &  \multicolumn{1}{c|}{$10^{-4}$} & \multicolumn{1}{c|}{26.28}& \multicolumn{1}{c|}{3.41}& \multicolumn{1}{c|}{0.99}& \multicolumn{1}{c|}{17.25}&  \multicolumn{1}{|c|}{} &  \multicolumn{1}{c|}{$10^{-4}$}  & \multicolumn{1}{c|}{152.49}& \multicolumn{1}{c|}{10.53}& \multicolumn{1}{c|}{0.99}& \multicolumn{1}{c|}{43.58}\\ \cline{2-6}\cline{8-12}
	\multicolumn{1}{|c|}{} &  \multicolumn{1}{c|}{$10^{-6}$} & \multicolumn{1}{c|}{19.12}& \multicolumn{1}{c|}{3.49}& \multicolumn{1}{c|}{0.99}& \multicolumn{1}{c|}{13.14}&  \multicolumn{1}{|c|}{} &  \multicolumn{1}{c|}{$10^{-6}$}  & \multicolumn{1}{c|}{17.01}& \multicolumn{1}{c|}{10.06}& \multicolumn{1}{c|}{0.99}& \multicolumn{1}{c|}{26.98}\\ \hline
	
	\multicolumn{6}{|c|}{($n_1$,$n_2$,$n_3$,$n_4$)=(20,20,20,20,10)} &
	\multicolumn{6}{|c|}{($n_1$,$n_2$,$n_3$,$n_4$)=(60,50,40,30)}\\ \hline
	
	\multicolumn{1}{|c|}{\multirow{3}{*}{24}} &  \multicolumn{1}{c|}{$10^{-2}$} & \multicolumn{1}{c|}{37.93}& \multicolumn{1}{c|}{0.88}& \multicolumn{1}{c|}{1.00}& \multicolumn{1}{c|}{45.56}& \multicolumn{1}{|c|}{\multirow{3}{*}{20}} &  \multicolumn{1}{c|}{$10^{-2}$}  & \multicolumn{1}{c|}{31.42}& \multicolumn{1}{c|}{2.78}& \multicolumn{1}{c|}{1.00}& \multicolumn{1}{c|}{31.16}\\ \cline{2-6}\cline{8-12}
	\multicolumn{1}{|c|}{} &  \multicolumn{1}{c|}{$10^{-4}$} & \multicolumn{1}{c|}{9.10}& \multicolumn{1}{c|}{0.92}& \multicolumn{1}{c|}{1.00}& \multicolumn{1}{c|}{15.86}&  \multicolumn{1}{|c|}{} &  \multicolumn{1}{c|}{$10^{-4}$}  & \multicolumn{1}{c|}{1.17}& \multicolumn{1}{c|}{2.76}& \multicolumn{1}{c|}{1.00}& \multicolumn{1}{c|}{9.39}\\ \cline{2-6}\cline{8-12}
	\multicolumn{1}{|c|}{} &  \multicolumn{1}{c|}{$10^{-6}$} & \multicolumn{1}{c|}{715.16}& \multicolumn{1}{c|}{0.91}& \multicolumn{1}{c|}{1.00}& \multicolumn{1}{c|}{15.63}&  \multicolumn{1}{|c|}{} &  \multicolumn{1}{c|}{$10^{-6}$}  & \multicolumn{1}{c|}{4.14}& \multicolumn{1}{c|}{2.79}& \multicolumn{1}{c|}{1.00}& \multicolumn{1}{c|}{9.48}\\ \hline
	\multicolumn{1}{|c|}{\multirow{3}{*}{48}} &  \multicolumn{1}{c|}{$10^{-2}$} & \multicolumn{1}{c|}{166.00}& \multicolumn{1}{c|}{1.95}& \multicolumn{1}{c|}{1.00}& \multicolumn{1}{c|}{270.56}& \multicolumn{1}{|c|}{\multirow{3}{*}{40}} &  \multicolumn{1}{c|}{$10^{-2}$}  & \multicolumn{1}{c|}{6.99}& \multicolumn{1}{c|}{7.52}& \multicolumn{1}{c|}{1.00}& \multicolumn{1}{c|}{31.81}\\ \cline{2-6}\cline{8-12}
	\multicolumn{1}{|c|}{} &  \multicolumn{1}{c|}{$10^{-4}$} & \multicolumn{1}{c|}{161.62}& \multicolumn{1}{c|}{1.93}& \multicolumn{1}{c|}{1.00}& \multicolumn{1}{c|}{40.63}&  \multicolumn{1}{|c|}{} &  \multicolumn{1}{c|}{$10^{-4}$}  & \multicolumn{1}{c|}{2.58}& \multicolumn{1}{c|}{7.32}& \multicolumn{1}{c|}{1.00}& \multicolumn{1}{c|}{20.07}\\ \cline{2-6}\cline{8-12}
	\multicolumn{1}{|c|}{} &  \multicolumn{1}{c|}{$10^{-6}$} & \multicolumn{1}{c|}{52.01}& \multicolumn{1}{c|}{1.93}& \multicolumn{1}{c|}{1.00}& \multicolumn{1}{c|}{21.71}&  \multicolumn{1}{|c|}{} &  \multicolumn{1}{c|}{$10^{-6}$}  & \multicolumn{1}{c|}{2.49}& \multicolumn{1}{c|}{7.22}& \multicolumn{1}{c|}{1.00}& \multicolumn{1}{c|}{20.22}\\ \hline
	\multicolumn{1}{|c|}{\multirow{3}{*}{72}} &  \multicolumn{1}{c|}{$10^{-2}$} & \multicolumn{1}{c|}{73.70}& \multicolumn{1}{c|}{3.10}& \multicolumn{1}{c|}{1.00}& \multicolumn{1}{c|}{102.90}& \multicolumn{1}{|c|}{\multirow{3}{*}{60}} &  \multicolumn{1}{c|}{$10^{-2}$}  & \multicolumn{1}{c|}{11.48}& \multicolumn{1}{c|}{9.83}& \multicolumn{1}{c|}{1.00}& \multicolumn{1}{c|}{48.08}\\ \cline{2-6}\cline{8-12}
	\multicolumn{1}{|c|}{} &  \multicolumn{1}{c|}{$10^{-4}$} & \multicolumn{1}{c|}{113.13}& \multicolumn{1}{c|}{3.06}& \multicolumn{1}{c|}{1.00}& \multicolumn{1}{c|}{70.13}&  \multicolumn{1}{|c|}{} &  \multicolumn{1}{c|}{$10^{-4}$}  & \multicolumn{1}{c|}{6.38}& \multicolumn{1}{c|}{9.80}& \multicolumn{1}{c|}{1.00}& \multicolumn{1}{c|}{38.97}\\ \cline{2-6}\cline{8-12}
	\multicolumn{1}{|c|}{} &  \multicolumn{1}{c|}{$10^{-6}$} & \multicolumn{1}{c|}{34.28}& \multicolumn{1}{c|}{3.03}& \multicolumn{1}{c|}{1.00}& \multicolumn{1}{c|}{36.72}&  \multicolumn{1}{|c|}{} &  \multicolumn{1}{c|}{$10^{-6}$}  & \multicolumn{1}{c|}{16.35}& \multicolumn{1}{c|}{9.76}& \multicolumn{1}{c|}{1.00}& \multicolumn{1}{c|}{30.38}\\ \hline
\end{tabular}}
\end{table}

\end{exmp}

In the following, we give a comparison with
the generalized eigenvalue decomposition (GEVD) method,
which is a classical one for computing tensor decompositions
when the rank $r \leq n_2$.  We refer to
\cite{Leurgans1993decomposition,Sanchez1990Tensorial}
for the work about the GEVD method.
Consider a cubic order tensor $\mF \in \bbC^{n_1 \times n_2 \times n_3}$
with $n_1 \geq n_2 \geq n_3$.
Suppose $\mF = U^{(1)} \circ U^{(2)} \circ U^{(3)}$
is a rank-$r$ decomposition and $r \le n_2$.
Assume its first and second decomposing matrices $U^{(1)}, U^{(2)}$
have full column ranks and the third decomposing matrix
$U^{(3)}$ does not have colinear columns.
Denote the slice matrices
\be  \label{gevd_f1f2}
 F_1  \, \coloneqq \,  \mathcal{F}_{1:r,1:r,1} ,\quad
 F_2  \, \coloneqq \,  \mathcal{F}_{1:r,1:r,2}.
\ee
One can show that
\begin{align}
F_1=U^{(1)}_{1:r,:} \cdot \mbox{diag}(U^{(3)}_{1,:}) \cdot (U^{(2)}_{1:r,:})^T, \quad
F_2=U^{(1)}_{1:r,:} \cdot \mbox{diag}(U^{(3)}_{2,:}) \cdot (U^{(2)}_{1:r,:})^T.
\end{align}
This implies that the columns of $(U^{(1)}_{1:r,r})^{-T}$
are generalized eigenvectors of the matrix pair $(F_1^{T}, F_2^T)$.
Consider the transformed tensor
\be  \label{gevd:fhat}
\hat{\mathcal{F}} \, = \, (U^{(1)}_{1:r,r})^{-1}\times_1 \mathcal{F}_{1:r,:,:},
\ee
For each $s=1,\ldots,r$,  the slice
$\hat{\mathcal{F}}_{s,:,:}=U^{(2)}_{:,s} \cdot (U^{(3)}_{:,s})^T$
is a rank-$1$ matrix.
The matrices $U^{(2)}$, $U^{(3)}$ can be obtained
by computing rank-$1$ decompositions
for the slices $\hat{\mathcal{F}}_{s,:,:}$.
After this is done, we can solve the linear system
\be   \label{gevd:least_squares}
U^{(1)} \circ U^{(2)} \circ U^{(3)} = \mF
\ee
to get the matrix $U^{(1)}$.
The following is the GEVD method for
computing cubic order tensor decompositions
when the rank $r \le n_2$.

\begin{alg}
(The GEVD method.)\,
\begin{itemize}

\item[Input:] A tensor $\mF \in \bbC^{n_1 \times n_2 \times n_3}$
with the rank $r \leq n_2$.
%
%
%
\item[1.] Formulate the tensor $\hat{\mathcal{F}}$ as in \eqref{gevd:fhat}.

\item[2.] For $s=1,\ldots,r,$ compute $U^{(2)}_{:,s}$, $U^{(3)}_{:,s}$ from
the rank-$1$ decomposition of the matrix $\hat{\mathcal{F}}_{s,:,:}$.

\item[3.] Solve the linear system \eqref{gevd:least_squares} to get $U^{(1)}$.

\item [Output:] The decomposing matrices $U^{(1)},U^{(2)},U^{(3)}.$
\end{itemize}
\label{algorithm:gevd}
\end{alg}

We compare the performance of Algorithm~\ref{alg:TD:r<=n2}
and Algorithm~\ref{algorithm:gevd}
for randomly generated tensors with the rank $r \le n_2$.
We generate $\mF = U^{(1)} \circ U^{(2)} \circ U^{(3)}$
such that each $U^{(i)} \in \mathbb{C}^{n_i \times r}$.
The entries of $U^{(i)}$ are randomly generated complex numbers.
Their real and imaginary parts are randomly generated,
obeying the Gaussian distribution.
For each case of $(n_1,...,n_m)$ and $r$,
we generate $20$ random instances of $\mF$.
Algorithm~\ref{algorithm:gevd} is implemented by the function
$cpd\_gevd$ in the software {\tt Tensorlab}.
All the tensor decompositions are computed correctly by both methods.
The average CPU time (in seconds) for Algorithm~\ref{alg:TD:r<=n2}
is denoted as {\tt time-gp},
while the average CPU time for the GEVD method
is denoted as {\tt time-gevd}.
%
%
The computational results are reported in Table~\ref{Table:TD}.
The numerical experiments show that
Algorithm~\ref{alg:TD:r<=n2} is more computationally efficient
than Algorithm~\ref{algorithm:gevd}.

\begin{table}[htp]
\caption{A comparison for the performance of
Algorithms~\ref{alg:TD:r<=n2} and \ref{algorithm:gevd}.}
\label{Table:TD}
\begin{tabular}{|l|c|r|r|} \hline
($n_1$,$n_2$,$n_3$)  & $r$   & {\tt time-gevd} & {\tt time-gp}  \\ \hline
(40,30,30) & 30 & 0.91 &0.29\\\hline
(50,50,50) & 50 & 4.77 &0.85\\\hline
(100,100,100) & 80 & 12.17 &5.54\\\hline
(150,150,150) & 100 & 79.85 &13.30\\\hline
(200,200,200) & 120 & 161.83 &25.71\\\hline
(250,250,250) & 140 & 285.03 &55.71\\\hline
(300,300,300) & 100 & 306.64 &61.38\\\hline
(400,400,400) & 180 & 934.15 &271.21\\\hline
(500,500,500) & 200 & 1688.98 &539.75\\\hline
\end{tabular}
\end{table}

\section{Conclusions}
\label{sec:con}.

This paper gives computational methods for
computing low rank tensor decompositions and approximations.
The proposed methods are based on generating polynomials.
For a generic tensor of rank $r\leq \mathrm{min}(n_1,N_3)$,
its tensor decomposition can be obtained by Algorithm~\ref{alg:TD:r<=n2} .
Under some general assumptions, we show that if
a tensor is sufficiently close to a low rank one,
then the low rank approximating tensor produced by Algorithm~\ref{alg:LRTA}
is quasi-optimal. Numerical experiments are
presented to show the efficiency of the proposed methods.

\bigskip \noindent
{\bf Acknowledgement}
Jiawang Nie is partially supported by the NSF grant DMS-2110780.
Li Wang is partially supported by the NSF grant DMS-2009689.


\end{document}